\documentclass[12pt,reqno]{amsart}

\usepackage{a4wide}

\pdfoutput=1

\usepackage[utf8]{inputenc}
\usepackage[english]{babel}
\usepackage{amsmath,amsfonts,amsthm,amssymb,amscd,amsbsy}
\usepackage[all]{xy}
\usepackage{graphicx}
\usepackage{euscript}
\usepackage{mathtext}
\usepackage{upgreek}
\usepackage{hyperref}
\hypersetup{
    colorlinks=true,
    linkcolor=blue,
    filecolor=magenta,      
    urlcolor=cyan,
}
\usepackage{euscript}
\usepackage{mathtools}
\usepackage{microtype}
\usepackage{setspace}
\usepackage{fancyhdr}
\usepackage{pict2e}
\usepackage{mathrsfs}
\usepackage[shortlabels]{enumitem}
\usepackage{stackengine}
\usepackage{changepage}
\usepackage{parskip}
\usepackage{tikz}
\usetikzlibrary{arrows}
\usepackage{lmodern}
\usepackage{csquotes}
\usepackage{chngcntr}
\usepackage{etoolbox}
\usepackage{expl3}
\usepackage{pgfkeys}
\usepackage{pgfopts}
\usepackage{xparse}
\usepackage{xstring}
\usepackage[
mark=o,
root-radius=.035cm,
edge-length=0.46cm,
indefinite-edge-ratio=2,
indefinite-edge={
draw=black,fill=white,thin,densely dashed}]{dynkin-diagrams}

\counterwithout{equation}{section}

\usepackage{xpatch}
\xapptocmd\normalsize{%
 \abovedisplayskip=10pt plus 1pt minus 3pt
 \abovedisplayshortskip=1pt plus 3pt
 \belowdisplayskip=10pt plus 2pt minus 3pt
 \belowdisplayshortskip=8pt plus 3pt minus 2pt
}{}{}

\numberwithin{equation}{section}

\usepackage[
backend=biber,
style=ieee,
sorting=nyt,
giveninits=true,
dashed=false,
maxbibnames=99
]{biblatex}
\DeclareSymbolFont{largesymbols}{OMX}{cmex}{m}{n}

\newtheoremstyle{theoremstyle}
  {0mm} % Space above
  {0mm} % Space below
  {\itshape} % Body font
  {} % Indent amount
  {\bfseries} % Theorem head font
  {.} % Punctuation after theorem head
  {.5em} % Space after theorem head
  {} % Theorem head spec (can be left empty, meaning `normal')

\theoremstyle{theoremstyle}
\newtheorem{proposition}{Proposition}
\newtheorem{lemma}[proposition]{Lemma}

\newtheorem{independentcorollary}[proposition]{Corollary}
\newtheorem{corollary}{Corollary}[proposition]
\newtheorem{fact}[proposition]{Fact}

\newtheorem*{conclusion}{Conclusion}

\newtheorem{corollaryinner}{Corollary}[proposition] % just the internal version
\newenvironment{sepcorollary}[1][]
 {% #1 is the cross reference label
  \if\relax\detokenize{#1}\relax
  \else
    \ifcsname #1-used\endcsname
      \expandafter\xdef\csname #1-used\endcsname{\the\numexpr\csname #1-used\endcsname+1}%
    \else
      \expandafter\gdef\csname #1-used\endcsname{1}%
    \fi
    \renewcommand{\thecorollaryinner}{\ref*{#1}.\csname #1-used\endcsname}%
  \fi
  \corollaryinner
 }
 {\endcorollaryinner}

\newtheoremstyle{examplestyle}
  {0mm} % Space above
  {0mm} % Space below
  {} % Body font
  {} % Indent amount
  {\bfseries} % Theorem head font
  {.} % Punctuation after theorem head
  {.5em} % Space after theorem head
  {} % Theorem head spec (can be left empty, meaning `normal')

\theoremstyle{examplestyle}
\newtheorem{remark}[proposition]{Remark}

\newtheorem{example}[proposition]{Example}
\newtheorem{definition}[proposition]{Definition}
\newtheorem*{observation}{Observation}

\newtheorem*{agreement}{Agreement}

\newtheoremstyle{remarks}
  {0mm} % Space above
  {0mm} % Space below
  {} % Body font
  {} % Indent amount
  {\bfseries} % Theorem head font
  {.} % Punctuation after theorem head
  {-.5em} % Space after theorem head
  {} % Theorem head spec (can be left empty, meaning `normal')

\theoremstyle{remarks}

\renewcommand{\Re}{\hspace{0.07em}\mathrm{Re}\hspace{0.07em}}
\renewcommand{\Im}{\hspace{0.08em}\mathrm{Im}\hspace{0.04em}}
\newcommand{\Ker}{\mathrm{Ker} \hspace{0.04em}}

\newcommand{\ord}{\mathrm{ord}}

\newcommand{\tr}{\mathrm{tr}}

\renewcommand{\mod}{\mathrm{mod} \hspace{0.2em}}

\newcommand{\Aut}{\mathrm{Aut} \hspace{0.06em}}
\newcommand{\Int}{\mathrm{Int}\hspace{0.06em}}

\newcommand{\Hom}{\mathrm{Hom}}

\newcommand{\Rl}{\mathbb{R}}
\newcommand{\Cx}{\mathbb{C}}

\newcommand{\mb}[1]{\mathbb{{#1}}}

\newcommand{\wt}[1]{\widetilde{{#1}}}

\newcommand{\set}[1]{\hspace{-0.8pt} \left \{ \hspace{0.03em} {#1} \hspace{0.03em} \right \}}
\newcommand{\bilin}[2]{\langle \hspace{0.1em} {#1} \hspace{0.1em} , \hspace{0.1em} {#2} \hspace{0.1em} \rangle \hspace{0.02em}}
\newcommand{\bbilin}[2]{\big \langle \hspace{0.1em} {#1} \hspace{0.1em} , \hspace{0.1em} {#2} \hspace{0.1em} \big \rangle \hspace{0.02em}}
\newcommand{\cross}[2]{\langle \hspace{0.1em} {#1} \hspace{0.1em} | \hspace{0.1em} {#2} \hspace{0.1em} \rangle \hspace{0.02em}}
\newcommand{\restr}[2]{{\left.\kern-\nulldelimiterspace #1 \vphantom{\big|} \right|_{#2}}}

\newcommand{\hast}{\hspace{0.04em} {\ast} \hspace{0.05em}}
\newcommand{\II}{I \hspace{-0.21em} I}

\newcommand{\Spin}{\mathrm{Spin}}
\newcommand{\Prin}{\mathrm{Prin}}
\newcommand{\Fib}{\mathrm{Fib}}
\newcommand{\GL}{\mathrm{GL}}

\newcommand{\SO}{\mathrm{SO}}
\newcommand{\U}{\mathrm{U}}
\newcommand{\SU}{\mathrm{SU}}

\newcommand{\hol}{\mathrm{hol}}

\newcommand\altxrightarrow[2][0pt]{\mathrel{\ensurestackMath{\stackengine%
  {\dimexpr#1-7.5pt}{\xrightarrow{\phantom{#2}}}{\scriptstyle\!#2\,}%
  {O}{c}{F}{F}{S}}}}
\newcommand{\isoto}{\altxrightarrow[1pt]{\sim}}

\newcommand{\mysetminusD}{\hbox{\tikz{\useasboundingbox (-0.5pt,-0.5pt) rectangle (5pt,8pt); \draw[line width=0.6pt,line cap=round] (3.5pt,-1.5pt) -- (0,7.25pt);}} \hspace{-1pt}}
\newcommand{\mysetminusT}{\mysetminusD}
\newcommand{\mysetminusS}{\hbox{\tikz{\draw[line width=0.45pt,line cap=round] (2pt,0) -- (-0.5pt,5pt);}} \hspace{1pt}}
\newcommand{\mysetminusSS}{\hbox{\tikz{\draw[line width=0.4pt,line cap=round] (1.5pt,0) -- (0,3pt);}}}

\newcommand{\mysetminus}{\mathbin{\mathchoice{\mysetminusD}{\mysetminusT}{\mysetminusS}{\mysetminusSS}}}
 
\makeatletter
\newcommand{\extp}{\@ifnextchar^\@extp{\@extp^{\,}}}
\def\@extp^#1{\mathop{\bigwedge\nolimits^{\!#1}}}
\makeatother

\newcommand{\coslice}[3]{\hspace{0.05em} {\raisebox{.2em}{$#1$}\hspace{-0.1em}\big/^{#2}\hspace{-0.15em}\raisebox{-.2em}{$#3$}} \hspace{0.04em}}

\newcommand{\Coslice}[3]{\hspace{0.05em} {\raisebox{.2em}{$#1$}\hspace{-0.1em}\Big/^{#2}\hspace{-0.5em}\raisebox{-.2em}{$#3$}} \hspace{0.05em}}

\makeatletter
\DeclareRobustCommand{\loplus}{\mathbin{\mathpalette\dog@lsemi{+}}}
\DeclareRobustCommand{\roplus}{\mathbin{\mathpalette\dog@rsemi{+}}}

\newcommand{\dog@rsemi}[2]{\dog@semi{#1}{#2}{-90,90}}
\newcommand{\dog@lsemi}[2]{\dog@semi{#1}{#2}{270,90}}
\newcommand{\dog@semi}[3]{%
  \begingroup
  \sbox\z@{$\m@th#1#2$}%
  \setlength{\unitlength}{\dimexpr\ht\z@+\dp\z@\relax}%
  \makebox[\wd\z@]{\raisebox{-\dp\z@}{%
    \begin{picture}(1,1)
    \linethickness{\variable@rule{#1}}
    \roundcap
    \put(0.5,0.5){\makebox(0,0){\raisebox{\dp\z@}{$\m@th#1#2$}}}
    \put(0.5,0.5){\arc[#3]{0.5}}
    \end{picture}%
  }}%
  \endgroup
}
\newcommand{\variable@rule}[1]{%
  \fontdimen8  
  \ifx#1\displaystyle\textfont3\else
    \ifx#1\textstyle\textfont3\else
      \ifx#1\scriptstyle\scriptfont3\else
        \scriptscriptfont3\relax
  \fi\fi\fi
}

\makeatother

\begin{document}

\title[The spinor and Weierstrass representations of surfaces]{The spinor and Weierstrass representations of surfaces in space}
\author{Ivan Solonenko}
\address{Department of Mathematics, King's College London, United Kingdom}
\email{ivan.solonenko@kcl.ac.uk}

\maketitle

\begin{abstract}
   In this paper, following Sullivan (\cite{sullivan}) and Kusner and Schmitt (\cite{kusner-schmitt}), we study conformal immersions of Riemann surfaces into the three-dimensional Euclidean space. Regarding such immersions as special bundle maps from the tangent bundle of the surface to the cotangent bundle of the 2-dimensional sphere, we generalize the classical Weierstrass representation of minimal surfaces to the case of arbitrary conformal immersions. We study how such an immersion gives rise to a spin structure on the surface together with a pair of spinors and how the immersion itself can be studied by means of these spinors.
\end{abstract}
   
\section{Introduction}

The Weierstrass representation of a minimal surface in $\Rl^3$ is a classical method that allows one to describe the surface -- and its differential-geometric properties like the mean and Gaussian curvatures and the Gauss map -- by means of a pair of functions on an open subset of the complex plane, one of which is holomorphic and the other one is meromorphic. Such a representation is local and depends on the choice of a local holomorphic coordinate on the surface (also known as isothermal coordinates, especially in earlier literature on the subject). 

In his 1989 paper \cite{sullivan}, D. Sullivan observed that the Weierstrass representation can be rendered coordinate-free and global by replacing the aforementioned pair of functions with a triple of holomorphic 1-forms on the surface. This triple can be thought of as a holomorphic bundle map from the tangent bundle of the surface to the cotangent bundle of the 2-sphere. The unique spin structure of $\mb{S}^2$ can be then pulled back along this bundle map to induce a spin structure on the surface together with a pair of holomorphic sections of the corresponding line bundle of spinors. All of the original formulas describing the surface and its geometry can be rewritten neatly in terms of these two spinors. Conversely, given a Riemann surface with a fixed spin structure and a (nondegenerate) pair of holomorphic spinors, one can write down an explicit integral formula involving these spinors that gives a (possibly periodic) minimal immersion of the surface into $\Rl^3$.

Later in their article \cite{kusner-schmitt}, R. Kusner and N. Schmitt observed that the minimality condition in Sullivan's paper can be relinquished. The resulting 'generalized Weierstrass representation' works for an arbitrary conformal immersion of a Riemann surface to $\Rl^3$. This way, holomorphic 1-forms and spinors become smooth $(1,0)$-forms and smooth spinors, respectively, and one has an added integrability condition that ensures that an abstract pair of spinors on $M$ induces a (possibly periodic) conformal immersion of $M$ to $\Rl^3$.

Both Sullivan's and Kusner and Schmitt's expositions of the subject are rather short on detail and barely contain any proofs. The present article is an attempt to rectify this issue. Following and improving on the papers of Sullivan, Kusner, and Schmitt, we give a detailed and thorough exposition of the generalized Weierstrass and spinor representations of Riemann surfaces in $\Rl^3$.

The paper is organized as follows. In Section \ref{classical} we review the classical theory of Weierstrass representations of minimal surfaces in $\Rl^3$. In Section \ref{technical} we discuss a number of technical points related to $\Cx P^1$, its canonical and tautological line bundles, and its Veronese embedding into $\Cx P^2$. In Section \ref{main_section} we define (generalized) Weierstrass representations of a Riemann surface and study their relation to conformal immersions of the surface to $\Rl^3$. We also investigate how the geometric properties of a conformal immersion (for instance, whether it is minimal) can be expressed via its associated Weierstrass representation and how the latter can be used to recover the immersion itself. Finally, in Section \ref{spin_section} we introduce the notion of a spinor representation of a Riemann surface and show that it is essentially equivalent to that of a Weierstrass representation. We then reimagine the salient points of Section \ref{main_section} in the language of spinor representations. As a side quest, we also give a formal categorical argument why a spin structure on a principal $\SO(2)$-bundle is the same as a so-called square root of the corresponding Hermitian line bundle.

\section{The classical Weierstrass representation}\label{classical}

We start with a brief review of the classical theory of Weierstrass representations. Let $M \subset \Rl^3$ be an oriented minimal surface. By passing to local isothermal coordinates, we can think of $M$ as the image of a conformal embedding $\upvarphi \colon U \to \Rl^3$, where $U$ is an open subset of $\Cx$ (we can actually allow $\upvarphi$ to be an immersion, not necessarily an embedding). If we write $z = u + iv$, the conformality and minimality conditions on $\upvarphi$ can be rewritten as\footnote{Throughout the paper, $\cross{-}{-}$ will stand for the standard bilinear form on $\Rl^n$ or $\Cx^n$.}:

\begin{itemize}
    \item ($\upvarphi$ is a conformal) $||\upvarphi_u|| = ||\upvarphi_v||, \hspace{1pt} \cross{\upvarphi_u}{\upvarphi_v} = 0 \hspace{0.5em} \Leftrightarrow \hspace{0.5em} \cross{\upvarphi_z}{\upvarphi_z} = 0$, 
    \item ($\upvarphi$ is minimal $\Leftrightarrow$ harmonic) $\Updelta \hspace{0.7pt} \upvarphi = 0 \; \Leftrightarrow \; \upvarphi_{z \bar{z}} = 0$.
\end{itemize}

Now assume $U$ is simply connected. Then there is a smooth map $\widehat{\upvarphi} \colon U \to \Rl^3$ satisfying
$$
\left\{
\begin{aligned}
\widehat{\upvarphi}_u &= - \upvarphi_v, \\
\widehat{\upvarphi}_v &= \upvarphi_u,
\end{aligned}
\right.
$$
and it is unique up to a translation in $\Rl^3$. It is also a conformal minimal immersion (\textbf{adjoint} to $\upvarphi$). The conditions on $\widehat{\upvarphi}$ are taken to be as they are precisely so that the immersion $f = \upvarphi + i \widehat{\upvarphi} \colon U \to \Cx^3$ is holomorphic. It is also isotropic, meaning that $\cross{f'}{f'} = 0$. Conversely, given a holomorphic isotropic immersion of $U$ to $\Cx^3$, its real and imaginary parts are adjoint conformal minimal immersions of $U$ to $\Rl^3$. If we write $f' = (\Upphi_1, \Upphi_2, \Upphi_3)$, we have:
$$
\Upphi_1^2 + \Upphi_2^2 + \Upphi_3^2 = 0.
$$
This condition can be rewritten as
\begin{equation}\label{Phi_formula}
(\Upphi_1 + i\Upphi_2)(\Upphi_1 - i\Upphi_2) = -\Upphi_3^2.
\end{equation}
Casting aside the cases when $\Upphi_1 = i\Upphi_2$ on the entire $U$ (which is definitely not the case unless $\upvarphi$ is an immersion into a horizontal plane), we can rewrite this as
$$
\Upphi_1 + i\Upphi_2 = \frac{- \Upphi_3^2}{\Upphi_1 - i\Upphi_2}.
$$
We introduce a pair of functions on $U$: 
$$
\upmu = \Upphi_1 - i\Upphi_2, \quad \upnu = \frac{\Upphi_3}{\Upphi_1 - i\Upphi_2}.
$$
Clearly, $\upmu$ is holomorphic and $\upnu$ is meromorphic. Observe that the function $\upmu \upnu^2 = - \Upphi_1 - i\Upphi_2$ is holomorphic. What is more, since $f$ is an immersion, all $\Upphi_i$'s cannot vanish simultaneously at any point of $U$. Together with \eqref{Phi_formula}, this implies that if $p \in U$ is a zero of $\upmu$, it cannot also be a zero of $\upmu\upnu^2$. We conclude that $\upmu$ and $\upnu$ are related by the following condition: the set of zeroes of $\upmu$ coincides with the set of poles of $\upnu$, and given any point $p$ in this set, $\ord_p (\upmu) = -2 \ord_p (\upnu)$.

We can recover $f'$ using these two functions:
$$
\Upphi_1 = \frac{\upmu}{2}(1-\upnu^2), \quad \Upphi_2 = \frac{i\upmu}{2}(1+\upnu^2), \quad \Upphi_3 = \upmu \upnu.
$$
The function $f'$, in turn, determines $f$ (and hence $\upvarphi$ and $\widehat{\upvarphi}$) up to translations in $\Cx^3$ by means of integration along paths in $U$:
$$
f (p) = f(p_0) + \int_{p_0}^p f' dz,
$$
where $p_0 \in U$ is fixed. This integral does not depend on the choice of a path because the 1-form $f'dz = df$ is exact. We see that, up to a translation in $\Rl^3$, the initial immersion $\upvarphi$ can be expressed as the real part of an integral of some holomorphic 1-form (or rather a triple of those) written in terms of $\upmu$ and $\upnu$:
\begin{equation}\label{intform}
\upvarphi(p) = \upvarphi(p_0) + \Re \left(\int_{p_0}^p \frac{\upmu}{2}(1-\upnu^2) dz, \; \int_{p_0}^p \frac{i\upmu}{2}(1+\upnu^2) dz, \; \int_{p_0}^p \upmu \upnu dz \right).
\end{equation}
Formula \eqref{intform} is called the Weierstrass representation of $M$. Going in the opposite direction, one can start with a holomorphic function $\upmu$ and a meromorphic function $\upnu$ that are related by the aforementioned condition\footnote{One can actually relax this condition to the requirement that $\upmu\upnu^2$ be holomorphic, but then one has to consider more general (so-called ``branched'') minimal immersions: $\upvarphi$ is required to be harmonic, nonconstant, and conformal outside of its singular points (which will automatically be discrete).}: the zeroes of $\upmu$ are precisely the poles of $\upnu$ and, given any such zero/pole $p \in U$, $\ord_p (\upmu) = -2 \ord_p (\upnu)$. Using the same integral formula \eqref{intform}, one obtains a conformal minimal immersion $U \to \Rl^3$.

We will render this theory coordinate-free and generalize it to arbitrary conformal immersions of Riemann surfaces into $\Rl^3$. Having done that, we will be able to reinterpret the theory completely in the language of spinors.

\section{The sphere, its canonical bundle, and the Veronese embedding}\label{technical}

Before we proceed to main part, we need to discuss a number of technical points regarding the canonical bundle of the sphere and the Veronese embedding.

\begin{agreement}
Whenever we have a map between smooth manifolds, it is assumed to be smooth, unless mentioned otherwise. Also, we are going to meet a lot of interplay between real and complex geometry, so let us agree on the following. Whenever $E$ is a complex vector space or complex vector bundle, $E^*$ will stand for its \textit{real} dual. Its complex dual will be denoted by $E^{* 1,0} \subseteq E^*_{\Cx}$. The two are surely isomorphic as complex vector spaces/bundles by means of taking the real part inside $E^*_{\Cx}$, where the complex structure on $E^*$ is taken to be dual to the one on $E$. We often suppress this isomorphism from the notation. In case we want to stress whether the functionals being considered are real- or complex-valued, we will stick to this notation to avoid confusion. The conjugate $E^{* 0,1} = \overline{\displaystyle E^{* 1,0}} \subseteq E^*_{\Cx}$ is the antidual to $E$ (the space/bundle of $\Cx$-antilinear functionals).
\end{agreement}

The first technical moment is how one identifies $\mathbb{S}^2$ with $\Cx P^1$. There is a bunch of similar ways to do this, all of them use the stereographic projection, but there is a degree of freedom in picking a pole of $\mathbb{S}^2$ and an affine chart on $\Cx P^1$. We choose the stereographic projection from the North pole $N = (0, 0, 1)$:
$$
\uptau \colon \mathbb{S}^2 \mysetminus \set{N} \isoto \Cx, \; (x, y, z) \mapsto \frac{x+iy}{1-z}.
$$
Then we identify $\Cx$ with the affine chart $U_0 = \set{[z_0 \colon z_1] \mid z_0 \ne 0} = \Cx P^1 \mysetminus \set{[0:1]} \subset \Cx P^1$ and send $N$ to the missing point $[0:1]$. The resulting map $\mathbb{S}^2 \isoto \Cx P^1$ is given by
$$
(x, y, z) \mapsto \begin{cases}
[1-z:x+iy], &\text{if} \hspace{0.4em} z \ne 1, \\
[0:1], &\text{if} \hspace{0.4em} z = 1,
\end{cases}
$$
and it is a biholomorphism when the sphere is oriented by a vector field pointing \textit{inside} the unit ball (this is \textit{the opposite} of the boundary orientation of $\partial \mathbb{B}^2 = \mathbb{S}^2$). The inverse mapping looks like this:
$$
\begin{cases} [1: z] \mapsto \left( \dfrac{2 \Re z}{|z|^2 + 1}, \dfrac{2 \Im z}{|z|^2 + 1}, \dfrac{|z|^2 - 1}{|z|^2 + 1} \right), \\
[0:1] \mapsto N. \end{cases}
$$
The projective line, in turn, can be identified with the standard conic in $\Cx P^2$ by means of the Veronese embedding:
$$
\Cx P^1 \hookrightarrow \Cx P^2, \; [z_0:z_1] \mapsto [z_0^2:z_0 z_1:z_1^2].
$$
The image of this map is cut out be the equation $w_0 w_2 = w_1^2$. It is a unique smooth plane conic up to projective automorphisms of $\Cx P^2$. It would be convenient for us to apply such an automorphism, i.e. linearly change the coordinates on $\mathbb{C}^3$:
$$
\left\{
\begin{aligned}
\widetilde{w}_0 &= w_0 - w_2;\\
\widetilde{w}_1 &= i(w_0 + w_2);\\
\widetilde{w}_2 &= 2w_1.
\end{aligned}
\right.
$$
In these new coordinates, which we are going to denote $w_0, w_1, w_2$ from now on, our conic becomes $[Q] = \set{[w_0:w_1:w_2] \in \Cx P^2 \mid w_0^2 + w_1^2 + w_2^2 = 0}$. Let us denote the affine cone over it by $Q \subset \mathbb{C}^3$. It can also be described as the set of isotropic vectors of the standard nondegenerate quadratic form on $\mathbb{C}^3$. In these new coordinates, the Veronese embedding is given by:
$$
\tilde{\uptheta} \colon \Cx P^1 \isoto [Q] \subset \Cx P^2, \; [z_0:z_1] \mapsto [z_0^2 - z_1^2:i(z_0^2 + z_1^2):2z_0 z_1].
$$
This is the projectivization of the map
$$
\uptheta \colon \Cx^2 \to \Cx^3, \; (z_0,z_1) \mapsto (z_0^2 - z_1^2, i(z_0^2 + z_1^2), 2z_0 z_1).
$$
The image of $\uptheta$ is $Q$, and, when restricted to the set of nonzero vectors, $\uptheta$ gives a two-sheeted holomorphic covering $\Cx^2 \mysetminus \set{0} \twoheadrightarrow Q \mysetminus \set{0}$. We are going to use this map to relate the tautological bundles over $\Cx P^1$ and $[Q]$ to their canonical bundles.

First of all, observe that the tautological line bundle $\mathcal{O}_{\Cx P^1}(-1)$ over $\Cx P^1$ is a line subbundle of the trivial rank-2 bundle $\Cx P^1 \times \Cx^2$. If we restrict the projection of the latter onto $\Cx^2$ to $\mathcal{O}_{\Cx P^1}(-1)$, we obtain the blow-up of $\Cx^2$ at the origin. Altogether, we have the following commutative diagram:
$$
\xymatrix{
\mathcal{O}_{\Cx P^1}(-1) \mysetminus \set{0} \ar[r]_(0.59){\sim} \ar@{^{(}->}[d] & \Cx^2 \mysetminus \set{0} \ar@{^{(}->}[d] \ar@/^1.5em/@{->>}[dd] \\
\mathcal{O}_{\Cx P^1}(-1) \ar[r] \ar@{->>}[dr] & \Cx^2 \\
& \Cx P^1}
$$
Here $\mathcal{O}_{\Cx P^1}(-1) \mysetminus \set{0}$ stands for the complement to the zero section. Similarly, the tautological line bundle $\mathcal{O}_{[Q]}(-1) = \restr{\mathcal{O}_{\Cx P^2}(-1)}{[Q]}$ is a line subbundle of the trivial bundle $[Q] \times \Cx^2$. When restricted to $\mathcal{O}_{[Q]}(-1)$, the projection of the latter onto $\Cx^2$ gives the blow-up of $Q$ at the origin (this is a resolution of an ordinary double point). This leads to the following diagram:
$$
\xymatrix{
Q \mysetminus \set{0} \ar@{^{(}->}[d] \ar@/_1.5em/@{->>}[dd] & \mathcal{O}_{[Q]}(-1) \mysetminus \set{0} \ar[l]^(0.57){\sim} \ar@{^{(}->}[d] \\
Q & \mathcal{O}_{[Q]}(-1) \ar[l] \ar@{->>}[dl] \\
[Q] }
$$
When we make a bridge between the two diagrams above by means of $\uptheta$ and $\tilde{\uptheta}$, we end up with the following commutative diagram:
$$
\xymatrix{
\mathcal{O}_{\Cx P^1}(-1) \mysetminus \set{0} \ar[r]_(0.59){\sim} \ar@{^{(}->}[d] & \Cx^2 \mysetminus \set{0} \ar@{^{(}->}[d] \ar@/^1.5em/@{->>}[dd]|\hole \ar@{->>}[r]^{\uptheta}& Q \mysetminus \set{0} \ar@{^{(}->}[d] \ar@/_1.5em/@{->>}[dd]|\hole & \mathcal{O}_{[Q]}(-1) \mysetminus \set{0} \ar[l]^(0.57){\sim} \ar@{^{(}->}[d] \\
\mathcal{O}_{\Cx P^1}(-1) \ar[r] \ar@{->>}[dr] & \Cx^2 \ar[r]^{\uptheta} & Q & \mathcal{O}_{[Q]}(-1) \ar[l] \ar@{->>}[dl] \\
& \Cx P^1 \ar[r]^{\tilde{\uptheta}}_{\sim} & [Q] }
$$
Now, $\uptheta$ is a homogeneous polynomial map, so it sends (linear) lines in $\Cx^2$ to lines in $Q$. This correctly defines a (nonlinear!) bundle map from $\upgamma \colon \mathcal{O}_{\Cx P^1}(-1) \to \mathcal{O}_{[Q]}(-1)$ covering $\tilde{\uptheta}$. As a map to $\Cx P^2 \times \Cx^3$, $\upgamma$ is just the product of $\mathcal{O}_{\Cx P^1}(-1) \twoheadrightarrow \Cx P^1 \xrightarrow{\tilde{\uptheta}} [Q]$ and  $\mathcal{O}_{\Cx P^1}(-1) \to \Cx^2 \xrightarrow{\uptheta} \Cx^3$, so it is holomorphic. Outside of the zero sections, it is just $\uptheta$ composed and precomposed with two biholomorphisms as in the upper part of the diagram above (so it is also a two-sheeted holomorphic covering map). On the fibers, $\upgamma$ looks like $\Cx \twoheadrightarrow \Cx, z \mapsto z^2$. But such a quadratic bundle map is the same as a linear bundle isomorphism (also covering $\tilde{\uptheta}$) $\widetilde{\upgamma} \colon \mathcal{O}_{\Cx P^1}(-1)^{\otimes 2} = \mathcal{O}_{\Cx P^1}(-2) \isoto \mathcal{O}_{[Q]}(-1)$, which sends $v \otimes v$ to $\upgamma(v)$. Since $\mathcal{O}_{\Cx P^1}(-2)$ is isomorphic to the canonical bundle of $\Cx P^1$, we arrive at the following:
\begin{conclusion}
There is a commutative square of isomorphisms of holomorphic line bundles over $\Cx P^1$ and $[Q]$:
$$
\xymatrix{
\mathcal{O}_{\Cx P^1}(-2) \ar[r]^{\widetilde{\upgamma}}_{\sim} \ar[d]_{\rotatebox{90}{$\sim$}} & \mathcal{O}_{[Q]}(-1) \ar[d]_{\rotatebox{90}{$\sim$}} \\
T^*\Cx P^1 & \ar[l]_-{d \tilde{\uptheta}^*}^-{\sim} T^*[Q]}
$$
Both horizontal maps cover $\tilde{\uptheta}$.
\end{conclusion}
It will also be vital for us to have an explicit description of the vertical isomorphisms in this diagram. It suffices to describe the left one. First we trivialize both $\mathcal{O}_{\Cx P^1}(-2)$ and $T^*\Cx P^1$ over $U_0$ and $U_1$ with the same cocycle. We trivialize $\mathcal{O}_{\Cx P^1}(-1)$ by means of the sections $\upsigma_0([z_0:z_1]) = (1,\tfrac{z_1}{z_0})$ and $\upsigma_1([z_0:z_1]) = (\tfrac{z_0}{z_1},1)$, which automatically trivializes $\mathcal{O}_{\Cx P^1}(-2)$ by $\upsigma_0 \otimes \upsigma_0$ and $\upsigma_1 \otimes \upsigma_1$. For $T^*\Cx P^1 \cong T^{* 1,0}\Cx P^1$, we take the section $-2d (\frac{z_1}{z_0})$ and $2d (\frac{z_0}{z_1})$. One can readily check that the transition function for both of these trivializations is $\uptau_{01} \colon U_0 \cap U_1 \to \Cx^{\times}, [z_0:z_1] \mapsto (\frac{z_0}{z_1})^2$. Thus, we can construct an isomorphism between our bundles just by sending the trivializing sections to each other:
$$
\mathcal{O}_{\Cx P^1}(-2) \isoto T^{* 1,0}\Cx P^1, \; (z_0, z_1) \otimes (z'_0, z'_1) \mapsto \begin{cases}
-2z_0 z'_0 d(\frac{z_1}{z_0})_{[z_0:z_1]}, &\text{if} \hspace{0.4em} [z_0:z_1] \in U_0; \\
2z_1 z'_1 d(\frac{z_0}{z_1})_{[z_0:z_1]}, &\text{if} \hspace{0.4em} [z_0:z_1] \in U_1,
\end{cases}
$$
where $(z_0, z_1)$ and $(z'_0, z'_1)$ are points on the line corresponding to $[z_0:z_1]$. Any other holomorphic isomorphism between these bundles differs from this one by a scalar because $\Cx P^1$ is compact. We make this particular choice in order to get rid of some coefficients later on (see Proposition \ref{hopf}).

\section{Conformal immersions and Weierstrass representations}\label{main_section}

In this section we define (generalized) Weierstrass representations of a Riemann surface and study their relation to conformal immersions of the surface into $\Rl^3$. We show that to every conformal immersion $M \to \Rl^3$ one can associate a Weierstrass representation of $M$, and the latter can be used to study the geometry of the immersion and even recover the immersion itself.

\subsection{Weierstrass representations}

We begin with the following general commutative diagram:
$$
\xymatrix{
\SO(2) \ar@{^{(}->}[r] & \Rl^+ \times \SO(2) \ar@{^{(}->}[r] & \GL^+(2, \Rl) \ar@{^{(}->}[r] & \GL(2, \Rl) \\
\U(1) \ar@{^{(}->}[r] \ar[u]^-{\rotatebox{90}{$\sim$}} & \GL(1, \Cx) \ar[u]^-{\rotatebox{90}{$\sim$}} \ar@{^{(}->}[ur] \\
\mathbb{T} \; \ar@{^{(}->}[r] \ar[u]^-{\rotatebox{90}{$\sim$}} & \Cx^{\times} \ar[u]^-{\rotatebox{90}{$\sim$}} }
$$
The right column tells us that if $V$ is a two-dimensional real vector space, to pick a complex structure on it (a reduction of the frame space $\mathrm{F}(V)$ from $\GL(2,\Rl)$ to $\GL(1, \Cx)$) is the same as to pick an orientation and a conformal class of Euclidean inner products (a reduction to $\Rl^+ \times \SO(2)$). Assume such a datum is chosen and fixed and let $I$ denote the complex structure. Then, for any $v \in V$, $v$ and $Iv$ are of the same length and they form an oriented orthogonal basis for $V$. In particular, $I$ is orthogonal with respect to any inner product within the given conformal class. The left column tells us that to choose a particular inner product within our class is the same as to choose a Hermitian inner product on the one-dimensional complex vector space $(V, I)$ -- both are further reductions of $\mathrm{F}(V)$ to $\SO(2) \cong \U(1)$. As an $(\Rl^+ \times \SO(2))$-set (resp., $\GL(1, \Cx)$-set), the reduction $\mathrm{F}_{\Rl^+ \times \SO(2)}(V)$ (resp., $\mathrm{F}_{\GL(1, \Cx)}(V)$) is isomorphic to $V \mysetminus \set{0}$ via the map $(v, Iv) \mapsto v$, where the action of $\Rl^+ \times \SO(2) \cong \GL(1, \Cx) \cong \Cx^\times$ on $V \mysetminus \set{0}$ is given simply by the scalar multiplication $\Cx^\times \curvearrowright V \mysetminus \set{0}$. Thus, a further $\SO(2)$- (or $\U(1)$-) reduction is just a choice of a $\mathbb{T}$-orbit in $V \mysetminus \set{0}$. In terms of the corresponding Euclidean and Hermitian inner products, this orbit is the circle of unit-length vectors. Finally, these Hermitian and Euclidean products $h$ and $g$ come from the same K\"{a}hler triple: $h(v,w) = g(v,w) + i\upomega(v,w)$, where $\upomega(v,w) = g(v,Iw)$ is the fundamental symplectic form of $g$ and $I$ (and also the volume form of $g$).\label{principal_isomorphic}

The discussion above applies almost verbatim when $V$ is replaced with a rank-2 real vector bundle $L$ over a smooth manifold $M$. The frame bundle $\mathrm{F}(L)$ is a principal $\GL(2, \Rl)$-bundle over $M$, its reduction to $\GL(1, \Cx) \cong \Rl^+ \times \SO(2)$ is the same as a choice of either an almost complex structure in $L$ or an orientation together with a conformal class of Euclidean bundle metrics (we will occasionally call this a $\Cx^{\times}$-reduction of $\mathrm{F}(L)$ or a $\Cx^{\times}$-structure in $L$). Its further reduction to $\SO(2) \cong \U(1)$ is a choice of a particular metric within the given conformal class or of a Hermitian bundle metric in $L$ thought of as a complex line bundle (a $\mathbb{T}$-reduction of $\mathrm{F}(L)$ or a $\mathbb{T}$-structure in $L$). This interplay between almost complex structures, bundle metrics, and orientations will pop up numerous times throughout the paper so it is worthwhile to have it laid out at the beginning.\label{reductions}

The special case of prime interest for us is when $M$ is two-dimensional and $L = TM$. When a $\Cx^{\times}$-reduction of $\mathrm{F}(TM)$ is specified, $M$ is a Riemann surface\footnote{This is because every almost complex structure on a 2-dimensional manifold is integrable by the Newlander-Nirenberg theorem.}. Note that in this case, as a complex line bundle, $TM$ is isomorphic to $T^{1,0}M$ by means of projection onto the latter along $T^{0,1}M$ inside $T_{\Cx}M$.

Before we define Weierstrass representations, we need one more observation that will prove useful throughout the paper.

\begin{observation}\label{obsforms}
Let $M$ be a smooth manifold and $E$ a complex vector bundle over $M$. A section of $\Hom_{\Rl}(E, M \times \Cx^m) \cong E^* \otimes_{\Rl} \Cx^m \cong (E^*_{\Cx})^{\oplus m}$ is the same as a smooth map $E \to \Cx^m$ that is $\Rl$-linear on each fiber. Such a section lies in $\Hom_{\Cx}(E, M \times \Cx^m) \cong E^{* 1,0} \otimes_{\Cx} \Cx^m \cong (E^{* 1,0})^{\oplus m}$ if and only if it is $\Cx$-linear on each fiber as a map $E \to \Cx^m$. If, in addition, $M$ is a complex manifold and $E$ is a holomorphic vector bundle, such a section is holomorphic if and only if it is holomorphic as a map $E \to \Cx^m$. A special case of this is when $M$ is a complex manifold and $E = TM$. A $\Cx^m$-valued 1-form on $M$ (an $m$-tuple of complex-valued 1-forms) is the same as a smooth map $TM \to \Cx^m$ that is $\Rl$-linear on each fiber. Such a form is of type (1,0) if and only if it is $\Cx$-linear on each fiber as a map $TM \to \Cx^m$, and holomorphic if and only if it is holomorphic as a map $TM \to \Cx^m$.
\end{observation}

Now we can finally proceed to defining what a Weierstrass representation is. It is going to be an object that can be assigned to an arbitrary Riemann surface $M$ as a piece of extra data. Moreover, any conformal immersion $M \to \Rl^3$ is going to determine such an object in a unique and well-defined way. This latter construction will serve as a motivation for the definition of a Weierstrass representation, so this is where we begin. 

Let $M$ be any Riemann surface. Let $\upvarphi = (X,Y,Z) \colon M \to \Rl^3$ be a smooth map. Then $\boldsymbol{\upomega} = 2 \partial \upvarphi = d\upvarphi - iI^*d\upvarphi$ is a smooth $\Cx^3$-valued (1,0)-form (just a triple of (1,0)-forms) on $M$ (here $\partial$ is the (1,0)-part of the de Rham differential $d$ on $M$).  Alternatively, $\upomega$ can be described as $TM \isoto T^{1,0}M \xrightarrow{2d\upvarphi} \Cx^3$: $\upomega(v) = d\upvarphi(v) - i d\upvarphi(Iv)$. In a local holomorphic coordinate $z$, $\upomega = 2\upvarphi_z dz$, where $\upvarphi_z = \frac{\partial \upvarphi}{\partial z}$. Since we have a fixed conformal class of Riemannian metrics on $M$ and the standard Riemannian metric on $\Rl^3$, it makes sense, for any $p \in M$, to ask whether $d\upvarphi_p \colon T_pM \to T_{\upvarphi(p)}\Rl^3$ is a conformal map. If this is the case for all the points where the differential $d\upvarphi$ does not vanish, we call $\upvarphi$ \textbf{pseudoconformal}. Note that this condition implies that $\upvarphi$ is automatically an immersion at all such points.

\begin{lemma}\label{conformal} Let $M$ be a Riemann surface and $\upvarphi \colon M \to \Rl^3$ a smooth map.
\begin{enumerate}
    \item $\upvarphi$ is pseudoconformal if and only if the image of $TM$ under $\upomega$ lies in $Q \subseteq \Cx^3$.
    \item $\upvarphi$ is a conformal immersion if and only if $\upomega$ does not vanish at any point and the image of $TM$ under $\upomega$ lies in $Q$.
\end{enumerate}
\end{lemma}

\begin{proof}
A straightforward computation. Let us write the components of $\upomega$ as $\upomega_i$. Then, given $p \in M$ and $v \in T_pM$ nonzero, we have
\begin{align*}
    \upomega_1(v)^2 + &\upomega_2(v)^2 + \upomega_3(v)^2 = \\
    &= (dX(v) - idX(Iv))^2 + (dY(v) - idY(Iv))^2 + (dZ(v) - idZ(Iv))^2 \\
    &= ||d\upvarphi(v)||^2 - 2i\cross{d\upvarphi(v)}{d\upvarphi(Iv)} - ||d\upvarphi(Iv)||^2.
\end{align*}
But then
\begin{align*}
    \upomega(v) \in Q \; &\Leftrightarrow \; \upomega_1(v)^2 + \upomega_2(v)^2 + \upomega_3(v)^2 = 0 \\
    &\Leftrightarrow \left\{ \begin{aligned} &||d\upvarphi(v)|| = ||d\upvarphi(Iv)|| \\
    &d\upvarphi(v) \perp d\upvarphi(Iv) \end{aligned} \right. \\
    &\Leftrightarrow \; \text{$\upvarphi$ is conformal at $p$ or $d\upvarphi$ vanishes at $p$.}
\end{align*}
This proves (1). For (2), just note that $\upvarphi$ is a conformal immersion $\Leftrightarrow$ $\upvarphi$ is pseudoconformal and $d\upvarphi$ does not vanish at any point $\Leftrightarrow$ the image of $TM$ under $\upomega$ lies in $Q$ and $\Re(\upomega)$ does not vanish. But the zeroes of a $(1,0)$-form coincide with the zeroes of its real part, for $\Im(\upomega) = -I^* \Re(\upomega)$. This completes the proof of the lemma.
\end{proof}

\begin{remark}
The same remains true for maps $M \to \Rl^n$ for any $n$ when $Q$ is defined to be the cone of isotropic vectors of the standard nondegenerate quadratic form on $\Cx^n$.
\end{remark}

As a consequence of Lemma \ref{conformal}, the pseudoconformality condition on $\upvarphi$ can be rewritten as an equation on $\upomega_i$'s in the space on quadratic differentials on $M$: 
\begin{equation}\label{quadratic_equation}
\upomega_1^2 + \upomega_2^2 + \upomega_3^2 = 0.
\end{equation}
\begin{observation}\label{obsbundle}
Let $M$ be a smooth manifold, $E$ a complex line bundle over $M$, $[P] \subseteq \Cx P^n$ a complex submanifold, $P \subseteq \Cx^{n+1}$ the affine cone over $[P]$, and, finally, $\mathcal{O}_{[P]}(-1)$ the tautological line bundle over $[P]$. Assume we have a smooth $\Cx$-linear bundle map from $E \to M$ to $\mathcal{O}_{[P]}(-1) \to [P]$ that is an isomorphism on each fiber (we will call such bundle maps nonvanishing). Then, projecting $\mathcal{O}_{[P]}(-1) \subseteq [P] \times \Cx^{n+1}$ onto $\Cx^{n+1}$, we obtain a smooth map $E \to \Cx^{n+1}$ that is $\Cx$-linear and nonzero on every fiber and whose image lies in $P$ (in other words, it maps every fiber $\Cx$-linearly onto a line in $P$). But this is the same as a smooth nonvanishing section of $(E^{* 1,0})^{\oplus (n+1)}$ mapping $E$ to $P$. Conversely, given such a section $\upomega \colon E \to \Cx^{n+1}$, we have 
$$
\xymatrix{
E \mysetminus \set{0} \ar[r]^-{\upomega} \ar@{->>}[d]^-{\uppi_E} & P \mysetminus \set{0} \ar@{->>}[d]^-{\uppi_P} \\
M \ar@{-->}[r]^-{\upeta} & [P]}.
$$
The left vertical arrow is a smooth surjective submersion (as is the right one, for that matter), so $\uppi_P \circ \upomega$ passes through it to give a smooth map $\upeta \colon M \to [P]$, and $(\upeta \circ \uppi_E, \upomega) \colon E \to \mathcal{O}_{[P]}(-1)$ is a smooth nonvanishing $\Cx$-linear bundle map covering $\upeta$. In particular, when $[P] = \Cx P^n$, nonvanishing $\Cx$-linear bundle maps from $E \to M$ to $\mathcal{O}_{\Cx P^n}(-1) \to \Cx P^n$ correspond simply to nonvanishing sections of $(E^{* 1,0})^{\oplus (n+1)}$. If $M$ is a complex manifold and $E$ is a holomorphic vector bundle, such a bundle map is holomorphic if and only if the corresponding section of $(E^{* 1,0})^{\oplus (n+1)}$ is holomorphic (by the observation on page \pageref{obsforms} and the fact that $\uppi_E$ is a holomorphic surjective submersion).
\end{observation}

With regard to Riemann surfaces, this gives us the following

\begin{independentcorollary}\label{weierstrass}
Let $M$ be a Riemann surface. Then one has the following bijections:
\begin{multline*}
\resizebox{0.97\hsize}{!}{
$\left\{
\begin{gathered}
\text{nonvanishing triples} \\
(\upomega_i) \; \text{of (1,0)-forms} \\ 
\text{s.\hspace{1pt}th.} \;\sum\nolimits_{i=1}^3{\upomega_i^2} = \hspace{0.1em} 0\\
\end{gathered}
\right\} \rightleftarrows \left\{
\begin{gathered}
\text{maps} \; TM \xrightarrow{\upomega} \Cx^3 \; \text{that} \\
\text{map each fiber $\Cx$-linearly} \\ 
\text{onto a line in $Q$}\\
\end{gathered}
\right\} \rightleftarrows \left\{
\begin{gathered}
\text{nonvanishing $\Cx$-linear bundle} \\
\text{maps} \; (\upomega, \upeta) \; \text{from} \; TM \to M \\ 
\text{to} \; \mathcal{O}_{[Q]}(-1) \to [Q] \\
\end{gathered} \right\}$
}
\end{multline*}
\end{independentcorollary}

And as we already know, the latter are the same as nonvanishing bundle maps to $(T^*[Q] \to [Q]) \cong (T^*\Cx P^1 \to \Cx P^1) \cong (T^*\mathbb{S}^2 \to \mathbb{S}^2)$.

\begin{definition}
Given a Riemann surface $M$, we call an element of any of the three sets in Corollary \ref{weierstrass} a \textbf{Weierstrass representation of} $\boldsymbol{M}$.
\end{definition}

\begin{conclusion}
Lemma \ref{conformal} simply says that any conformal immersion of $M$ to $\Rl^3$ induces a Weierstrass representation of $M$.
\end{conclusion}

\subsection{The Gauss map}

We are going to investigate how a conformal immersion $\upvarphi \colon M \to \Rl^3$ and its geometrical properties can be studied by means of the corresponding Weierstrass representation $\upomega$. In this subsection we give a geometric interpretation to the map $\upeta \colon M \to [Q]$ induced by $\upomega$. 

First of all, we establish some notation that will facilitate a lot of proofs in the sequel. Let $M$ be a Riemann surface, and let $\upomega$ be any Weierstrass representation of $M$. In a similar vein to what we did in Section \ref{classical}, let us rewrite the quadratic equation \eqref{quadratic_equation} as:
\begin{equation}\label{omegaequation}
(\upomega_1 + i \upomega_2)(\upomega_1 - i \upomega_2) = -\upomega_3^2.
\end{equation}
The idea is that, since $\upomega_1, \upomega_2$, and $\upomega_3$ are related by an equation, some data are redundant and can be dispensed with. Denote $\boldsymbol{\upmu_+} = \upomega_1 + i\upomega_2$ and $\boldsymbol{\upmu_-} = \upomega_1 - i\upomega_2$. Let $\boldsymbol{M_+}, \boldsymbol{M_-} \subseteq M$ be their respective zero loci. Since $\upomega$ is nowhere-vanishing, these are disjoint closed sets, and their union is exactly the zero locus of $\upomega_3$. Let us denote the complements of $M_+$ and $M_-$ by $\boldsymbol{U_+}$ and $\boldsymbol{U_-}$, respectively. The $(1,0)$-form $\upmu_{\pm}$ is a nowhere-vanishing smooth section of the complex line bundle $T^{* 1,0}M$ over $U_{\pm}$. Therefore, there exists a unique smooth complex-valued function $\boldsymbol{\upnu_{\pm}}$ on $U_{\pm}$ such that $\upnu_{\pm}\upmu_{\pm} = \upomega_3$ (slightly informally, $\upnu_{\pm} = \frac{\upomega_3}{\upmu_{\pm}}$). Together, $\upmu_{\pm}$ and $\upnu_{\pm}$ recover $\upomega$ over $U_{\pm}$:
\begin{equation}\label{munurecover}
\restr{\upomega}{U_{\pm}} = \left( \frac{1-\upnu_{\pm}^2}{2}, \pm \frac{i(1+\upnu_{\pm}^2)}{2}, \upnu_{\pm} \right)\upmu_{\pm}.
\end{equation}
We also have the following relations:
$$
\upmu_+ \upmu_- = - \upomega_3^2; \quad \upnu_+ \upnu_- = -1 \; \text{(over $U_+ \cap U_-$)}; \quad \nu_{\pm}^{-1}(0) = M_{\mp}.
$$
Both pairs $(\upmu_+, \upnu_+)$ and $(\upmu_-, \upnu_-)$ are equally useful and none of them can recover the Weierstrass representation in general, for the set $M_{\pm}$ of ``poles'' (see p.~\pageref{meromorphic}) of $\upnu_{\pm}$ can be arbitrarily large. If either $M_+$ or $M_-$ has empty interior, \eqref{munurecover} recovers $\upomega$ entirely by continuity. For instance, as we will see in Proposition~\ref{munu}, this is the case when $\upomega$ comes from a minimal immersion. 

\begin{agreement}
For consistency, we will mostly use the pair $(\upmu_-, \upnu_-)$, so in the proofs we will write it as $(\boldsymbol{\upmu}, \boldsymbol{\upnu})$ to simplify the notation.
\end{agreement}

For every $p \in M$, $d\upvarphi(T_pM)$ is an oriented plane in $T_{\upvarphi(p)}\Rl^3 \cong \Rl^3$, so the normal line $d\upvarphi(T_pM)^{\perp}$ is naturally oriented as well (in such a way that the orientations in the decomposition $T_{\upvarphi(p)}\Rl^3 = d\upvarphi(T_pM) \oplus d\upvarphi(T_pM)^{\perp}$ match). This allows us to choose a positively oriented unit section\footnote{Which is a unit normal vector field along $\upvarphi(M)$ in case $\upvarphi$ is an embedding (hence always locally).} of $TM^{\perp} \subseteq \upvarphi^*T\Rl^3 \cong M \times \Rl^3$, which can be thought of as a smooth map $M \to \mathbb{S}^2$. It is called the \textbf{Gauss map of} $\boldsymbol{\upvarphi}$. Using the notation established above, we are ready to prove the following important

\begin{proposition}\label{gauss}
Let $\upvarphi$ be a conformal immersion of $M$ to $\Rl^3$ and $(\upomega, \upeta)$ its Weierstrass representation. Under the identifications $\mathbb{S}^2 \cong \Cx P^1 \cong [Q]$, $\upeta \colon M \to [Q]$ is exactly the Gauss map of $\upvarphi$.
\end{proposition}

\begin{proof}
\textsc{Step 1:} \textit{The Gauss map and $\upeta$ agree on $M_-$.} The North pole $N$ of the sphere goes to $[0:1]$ in $\Cx P^1$ and hence to $[-1:i:0] = [i:1:0]$ in $[Q]$ under our identifications. But the preimage of this point by $\upeta$ in exactly $M_-$. So we need to show that the Gauss image of a point is $N$ if and only if it lies in $M_-$. Let $p \in M$ and $v \in T_pM \mysetminus \set{0}$. We compute:
\begin{multline*}
p \in M_- \; \Leftrightarrow \; \upmu(p) = 0 \; \Leftrightarrow \; \upomega_1(v) = i\upomega_2(v) \; \Leftrightarrow \; \left\{
\begin{aligned}
&dX(Iv) = -dY(v);\\
&dY(Iv) = dX(v);\\
&dZ(v) = dZ(Iv) = 0
\end{aligned} \; \Leftrightarrow 
\right. \\[0.5em]
\Leftrightarrow \; \begin{aligned}
&d\upvarphi(Iv) \; \text{can be obtained from} \; d\upvarphi(v) \; \text{by the} \\
&\text{clockwise rotation around the z-axis by $90^{\circ}$}
\end{aligned} \; \Leftrightarrow \; \text{the Gauss image of $p$ is $N$.} 
\end{multline*}
\textsc{Step 2:} \textit{The Gauss map and $\upeta$ agree on $U_-$.} It is easy to see that, as a map to $\Cx P^1$, $\restr{\upeta}{U_-}$ is just $[1:\upnu]$. Therefore, as a map to $\mathbb{S}^2$, it is $\left(\frac{2\Re\upnu}{|\upnu|^2+1}, \frac{2\Im\upnu}{|\upnu|^2+1}, \frac{|\upnu|^2-1}{|\upnu|^2+1} \right)$. So we need to show that the Gauss map is given by this formula over $U_-$. Pick any $p \in U_-$ and nonzero $v \in T_pM$. First of all, we express the lengths of $d\upvarphi(v)$ and $d\upvarphi(Iv)$ via $\upnu$ and $\upmu$:
\begin{align*}
||d&\upvarphi(v)||^2 = ||d\upvarphi(Iv)||^2 = \frac{1}{2} ||\upomega(v)||^2 = \frac{1}{2} \sum_{i=1}^3 |\upomega_i(v)|^2 \\
&= \frac{1}{2}\left(\frac{(1-\upnu(p)^2)(1-\overline{\upnu}(p)^2)}{4}|\upmu(v)|^2 + \frac{(1+\upnu(p)^2)(1+\overline{\upnu}(p)^2)}{4}|\upmu(v)|^2 + |\upnu(p)|^2 |\upmu(v)|^2\right) \\
&= \frac{|\upmu(v)|^2}{8}\left(1-\upnu(p)^2 -\overline{\upnu}(p)^2 + |\upnu(p)|^4 + 1+\upnu(p)^2 +\overline{\upnu}(p)^2 + |\upnu(p)|^4 + 4|\upnu(p)|^2\right) \\
&= \frac{|\upmu(v)|^2}{8}\left(2 + 4|\upnu(p)|^2 + 2|\upnu(p)|^4\right) = \left(\frac{|\upmu(v)|(1+|\upnu(p)|^{2})}{2}\right)^2.
\end{align*}
Next, the cross-product of $d\upvarphi(v)$ and $d\upvarphi(Iv)$ is
\begin{multline*}
d\upvarphi(v) \times d\upvarphi(Iv) = \left(\begin{gathered}
dY(v)dZ(Iv) - dY(Iv)dZ(v) \\
- dX(v)dZ(Iv) + dX(Iv)dZ(v) \\
dX(v)dY(Iv) - dX(Iv)dY(v) \end{gathered} \right) = \Im \left(\begin{gathered}
\upomega_2(v) \hspace{0.04em} \overline{\upomega_3(v)} \\
\overline{\upomega_1(v)} \hspace{0.04em} \upomega_3(v) \\
\upomega_1(v) \hspace{0.04em} \overline{\upomega_2(v)} \end{gathered}\right) \\[0.3em]
= \frac{|\upmu(v)|^2}{4} \Im \left(\begin{gathered}
2i(1+\upnu(p)^2)\overline{\upnu(p)} \\
2(1-\overline{\upnu(p)}^2)\upnu(p) \\
i(\upnu(p)^2-1)(\overline{\upnu(p)^2}+1) \end{gathered}\right) = \frac{|\upmu(v)|^2}{4} \left(\begin{gathered}
2\Re(\overline{\upnu(p)}+\upnu(p)^2\overline{\upnu(p)}) \\
2\Im(\upnu(p)-\overline{\upnu(p)}^2\upnu(p)) \\
\Re(|\upnu(p)|^4 + \upnu(p)^2 - \overline{\upnu(p)}^2 -1) \end{gathered}\right) \\[0.3em]
= \frac{|\upmu(v)|^2}{4} \left(\begin{gathered}
2\Re\upnu(p)(1+|\upnu(p)|^2) \\
2\Im\upnu(p)(1+|\upnu(p)|^2) \\
|\upnu(p)|^4 -1 \end{gathered}\right) = \frac{|\upmu(v)|^2(1+|\upnu(p)|^2)}{4} \left(\begin{gathered}
2\Re\upnu(p) \\
2\Im\upnu(p) \\
|\upnu(p)|^2 -1 \end{gathered}\right).
\end{multline*}
Dividing this by $||d\upvarphi(v) \times d\upvarphi(Iv)|| = ||d\upvarphi(v)||^2$ yields the Gauss image of $p$:
$$
\left(\frac{2\Re\upnu(p)}{|\upnu(p)|^2+1}, \frac{2\Im\upnu(p)}{|\upnu(p)|^2+1}, \frac{|\upnu(p)|^2-1}{|\upnu(p)|^2+1} \right),
$$
which agrees with $\upeta(p)$. This completes the proof.
\end{proof}

\begin{agreement}
Given any Weierstrass representation $(\upomega, \upeta)$ of $M$, when thinking of $\upeta$ as a map to $\Cx P^1$ or $\mathbb{S}^2$, we will use the same letter. In particular, the Gauss map of a conformal immersion $M \to \Rl^3$ will be denoted by $\upeta$ from now on.
\end{agreement}

\begin{remark}\label{nuextension}
Given a Weierstrass representation $\upomega$ of $M$, $\restr{\upeta}{U_-} \colon U_- \to \Cx P^1$ is given simply by $[1: \upnu_-]$. This means that we can extend $\upnu_-$ to the whole $M$ by looking at it as a function $U_- \to \Cx \cong U_0 \subset \Cx P^1$ and defining it to be $[0:1]$ on $M_-$ (extending to the ``poles'' by ``infinity'' -- just as we do for meromorphic functions; cf. p. \pageref{meromorphic}). The resulting map $M \to \Cx P^1$ is smooth, for it is just another manifestation of $\upeta$. In a similar vein, we can extend $\upnu_+$ to the whole $M$ by sending $M_+$ to the ``infinity'' $[0:1]$. The resulting map $M \to \Cx P^1$ is also smooth, but it differs from $\upeta$ by the involution $R = \bigl[ \begin{smallmatrix} 0 & 1 \\ -1 \;\; & 0 \end{smallmatrix} \bigr] \in \mathrm{PGL}(2, \Cx)$ of $\Cx P^1$ because $\restr{\upeta}{U_+} = [-\upnu_+:1]$. In terms of $\mathbb{S}^2$, $R$ is the reflection in the y-axis.
\end{remark}

\subsection{The minimal case}

Now we turn our attention to the question of minimality. First, we need a couple of standard facts on harmonicity and the Hodge star operator:

\begin{fact}
Let $(M, g)$ be an oriented Riemannian $n$-manifold, and let $f$ be a smooth positive function on $M$. Then the Hodge star operators $\Upomega^k_{\Cx}(M) \isoto \Upomega^{n-k}_{\Cx}(M)$ and codifferentials $\Upomega^k_{\Cx}(M) \to \Upomega^{k-1}_{\Cx}(M)$ of the metrics $g$ and $fg$ are related as follows:

\begin{itemize}
    \item $\hast_{fg} = f^{k-\frac{n}{2}}\hast_g$.
    \item $d^*_{fg} = \frac{1}{f}d^*_g + \frac{(-1)^{n(k+1)+1}(k-\frac{n}{2})}{f}\hast_g(df\wedge\hast_g(-))$.
\end{itemize}
\end{fact}

\begin{corollary}\label{rshodge}
Let $M$ be a Riemann surface. Pick any Riemannian metric within its conformal class.
\begin{enumerate}
    \item The Hodge star operator $\hast$ is a conformally invariant automorphism of $\Upomega^1_{\Cx}(M)$, i.e. it does not depend on the choice of a metric within the conformal class (when restricted to 1-forms only). In fact, it is easy to see by taking orthonormal oriented frames that $\restr{\hast}{\Upomega^1_{\Cx}(M)}$ coincides with $-I^*$.
    \item A conformal factor $f$ changes the codifferential $d^*_g \colon \Upomega^1_{\Cx}(M) \to C^{\infty}_{\Cx}(M)$ (and hence the Laplacian $\Updelta_g = d^*_g \circ d$ on $C^{\infty}_{\Cx}(M)$) by a factor of $\frac{1}{f}$. Consequently, the notion of harmonicity is a conformal invariant for smooth functions on $M$. In fact, $f$ is harmonic if and only if both $\Re f$ and $\Im f$ are harmonic, and a real function is harmonic if and only if it is locally the real (or imaginary) part of some holomorphic function on $M$.
\end{enumerate}
\end{corollary}
    
It will be useful for us to look at the notions of harmonicity of holomorphicity for functions and 1-forms on Riemann surfaces in slightly more detail. It will pay off further on and render some statements almost obvious. As always, let $M$ be a Riemann surface.

\begin{enumerate}
    \item The notion of coclosedness and hence harmonicity for 1-forms on $M$ is also a conformal invariant. Indeed, $d^*_{fg}\upeta = \frac{1}{f} d^*_g\upeta = 0 \Leftrightarrow d^*_g\upeta = 0$. A 1-form $\upeta$ is coclosed if and only if $\hast\upeta$ is closed.
    \item A function $f$ is harmonic if and only if $df$ is harmonic. Therefore, a 1-form is harmonic if and only if it is locally the differential of a harmonic function.
    \item The projection $\prod^{1,0} \colon \Upomega^1_{\Cx}(M) \to \Upomega^{1,0}(M)$ sends $\upeta$ to $\frac{1}{2}(\upeta + i\hast\upeta)$. Similarly, $\prod^{0,1}\upeta = \frac{1}{2}(\upeta - i\hast\upeta)$. Consequently, if $f$ is a function, $\partial f = \frac{1}{2}(df + i\hast df), \overline{\partial} f = \frac{1}{2}(df - i\hast df)$. All of this follows from Corollary \ref{rshodge}(1).
    \item A 1-form $\upeta$ is of type (1,0) (resp., (0,1)) if and only if $\Im(\upeta) = \hast \Re(\upeta)$ (resp., $\Im(\upeta) = -\hast \Re(\upeta)$).
    \item A function $f$ is holomorphic if and only if $\Im(df) = \hast \Re(df)$.
    \item Let $\upeta$ be a (1,0)- or (0,1)-form. The following are equivalent:
    \begin{enumerate}[(i)]
        \item $\upeta$ is closed.
        \item $\upeta$ is coclosed.
        \item $\Re(\upeta)$ is harmonic.
        \item $\Im(\upeta)$ is harmonic.
        \item $\upeta$ is harmonic.
    \end{enumerate}
    This is a simple consequence of the previously observed fact that $\upeta$ is closed if and only if $\hast \upeta$ is coclosed.
    \item (1,0)-forms satisfying conditions (i)-(v) above are precisely holomorphic 1-forms, since $d = \partial$ on (1,0)-forms.
\end{enumerate}

Now we are ready to prove

\begin{proposition}\label{holomorphic}
Let $M$ be a Riemann surface, $\upvarphi \colon M \to \Rl^3$ a conformal immersion, and $\upomega = 2\partial\upvarphi \in \Upomega^{1,0}(M, \Cx^3)$ its Weierstrass representation. The following are equivalent:

\begin{enumerate}[(i)]
    \item $\upvarphi$ is minimal;
    \item $\upvarphi$ is harmonic;
    \item $\upomega$ is holomorphic;
    \item The bundle map $(\upomega, \upeta)$ from $TM \to M$ to $\mathcal{O}_{[Q]}(-1) \to [Q]$ is holomorphic;
    \item The Gauss map $\upeta \colon M \to \mathbb{S}^2$ of $\upvarphi$ is holomorphic.
\end{enumerate}
\end{proposition}

\begin{proof}
It is a standard fact that a submanifold of a Riemannian manifold is minimal if and only if its embedding into the ambient manifold is harmonic (see, e.g., \cite{xin}). Since both properties are local, it remains true for isometric immersions, and since harmonicity is a conformal notion for surfaces, it remains true for conformal immersions of surfaces, so we have $(i) \Leftrightarrow (ii)$. But $\upomega = d\upvarphi + i\hast d\upvarphi$ is a (1,0)-form, so $\upomega$ is holomorphic $\Leftrightarrow \Re(\upomega) = d\upvarphi$ is harmonic $\Leftrightarrow \upvarphi$ is harmonic, which gives $(ii) \Leftrightarrow (iii)$. Next, $(iii) \Leftrightarrow (iv) \Rightarrow (v)$ follows from the observation on page \pageref{obsbundle}. For the remaining part $(v) \Rightarrow (ii)$ we refer to \cite[Prop. 1.2.4]{fujimoto}.
\end{proof}

\begin{remark}
If $\upomega$ is an arbitrary Weierstrass representation of $M$, not necessarily coming from a conformal immersion, then part $(iii) \Leftrightarrow (iv) \Rightarrow (v)$ of Proposition \ref{holomorphic} still remains true (with the same proof). We call a \textbf{Weierstrass representation} satisfying $(iii)-(iv)$ \textbf{holomorphic}. Observe that holomorphicity of $\upeta$ implies holomorphicity of $\upomega$ only when they come from a conformal immersion. As an example, take some $M$ with a holomorphic Weierstrass representation $(\upomega, \upeta)$ and multiply $\upomega$ by a nonvanishing complex-valued smooth function $f$ that is not holomorphic. The resulting Weierstrass representation $(f\upomega, \upeta)$ is not holomorphic, although $\upeta$ is.
\end{remark}

\begin{remark}
It is essential in part $(iv)$ of Proposition \ref{holomorphic} that the orientation on $\mathbb{S}^2$ is chosen the way we did in Section \ref{technical}. If it was the opposite of ours, the Gauss map of a minimal immersion would never be holomorphic and, on the other hand, the Gauss map of a conformal immersion whose image lies in $\mathbb{S}^2$ (and that is definitely not minimal) would be holomorphic.
\end{remark}

Let $\upomega$ be a holomorphic Weierstrass representation of $M$. Then $\upmu_+$ and $\upmu_-$ are holomorphic 1-forms on $M$ and $\upnu_+$ and $\upnu_-$ are holomorphic functions on $U_+$ and $U_-$, respectively. There are three distinct cases:

\begin{enumerate}
    \item $\upomega_3 = 0$ (if $\upomega$ comes from a conformal immersion $\upvarphi$, this is equivalent to $\upvarphi$ being an immersion to a horizontal plane). This can be divided into:
    \begin{enumerate}[(i)]
        \item $\upmu_+ = 0$. Then $M_+ = U_- = M, M_- = U_+ = \varnothing$, $\upmu_-$ does not vanish, and $\upnu_- = 0$. If $\upomega$ comes from a conformal immersion $\upvarphi$, this is equivalent to $\upvarphi$ being an immersion to a horizontal plane with the Gauss map always pointing downward ($\upeta(M) = S$).
        \item $\upmu_- = 0$. Then $M_+ = U_- = \varnothing, M_- = U_+ = M$, $\upmu_+$ does not vanish, and $\upnu_+ = 0$. If $\upomega$ comes from a conformal immersion $\upvarphi$, this is equivalent to $\upvarphi$ being an immersion to a horizontal plane with the Gauss map always pointing upward ($\upeta(M) = N$).
    \end{enumerate}
    \item\label{meromorphic} $\upomega_3 \ne 0$. Then the set $M_+ \sqcup M_-$ of its zeroes is discrete. A priori, $\upnu_{\pm}$ is a holomorphic function on $U_{\pm}$, but we have already seen in Remark \ref{nuextension} that it tends to infinity at the points of $M_{\pm}$, so it is actually a meromorphic function on $M$ (given by the same formula $\frac{\omega_3}{\mu_{\pm}}$), and $M_{\pm}$ is indeed the set of its poles. 
\end{enumerate}

\begin{proposition}\label{munu}
Let $M$ be a Riemann surface. The map
$$
\left\{ \begin{gathered} \text{Weierstrass} \\
\text{representations of $M$} \end{gathered} \right\} \longrightarrow \left\{ \begin{gathered} \text{pairs $(\upmu_-, \upnu_-)$ with $\upmu_- \in \Upomega^{1,0}(M)$} \\
\text{and $\upnu_- \in C^{\infty}_{\Cx}(M \mysetminus \upmu_-^{-1}(0))$} \end{gathered} \right\}
$$
restricts to a one-to-one correspondence between holomorphic Weierstrass representations $\upomega$ with $\upomega_1 \ne i\upomega_2$ and pairs $(\upmu_-, \upnu_-)$ with $\upmu_-$ nonzero holomorphic and $\upnu_-$ meromorphic on $M$ that satisfy the following property: if $p \in M$ is a zero of $\upmu_-$ of order $m$, then it is a pole of $\upnu_-$ of order $m/2$ (in particular, $m$ must be even). A similar statement is true for the map from the set of Weierstrass representations to the set of pairs $(\upmu_+, \upnu_+)$.
\end{proposition}

\begin{proof}
Let $\upomega$ be a holomorphic Weierstrass representation of $M$. Let $p \in M_-$ be any zero of $\upmu_-$. We already know it is a pole of $\upnu_-$. Since it is not a zero of $\upmu_+$, if follows from \eqref{omegaequation} that $\ord_p(\omega_3) = \frac{1}{2}\ord_p(\upmu_-)$ and hence $\ord_p(\upnu_-) = \ord_p (\frac{\upomega_3}{\upmu_-}) = -\frac{1}{2}\ord_p(\upmu_-)$. Conversely, let $\upmu_-$ be a nonzero holomorphic 1-form and $\upmu_-$ a meromorphic function satisfying the above condition on their zeroes and poles. Then \eqref{munurecover} gives a holomorphic $\Cx^3$-valued 1-form on $U_-$ sending $TU_-$ to $Q$. The zeroes-poles condition ensures that the extension of this form to the isolated points of $M_-$ (by the same formula) gives not just a meromorphic but holomorphic 1-form. The proof for the $+$-pairs is completely analogous.
\end{proof}

\subsection{Understanding an immersion via its Weierstrass representation}

As an example of how a conformal immersion can be studied in terms of its Weierstrass representation, we will show how to express its induced metric, Hopf differential, and mean and Gaussian curvatures.

Let $\upomega$ be a Weierstrass representation of a Riemann surface $M$. For any $p \in M$, $\upomega_p \colon T_pM \isoto \mathcal{O}_{[Q]}(-1)_{\upeta(p)}$ is an isomorphism. As a subbundle of the trivial bundle $[Q] \times \Cx^3$, $\mathcal{O}_{[Q]}(-1)$ is endowed with a Hermitian bundle metric. This metric can be pulled back by $\upomega$ to produce a Hermitian metric on $M$ that we call $h$. Given $v \in TM$, $h(v) = h(v,v) = ||\upomega(v)||^2 = |\upomega_1(v)|^2 + |\upomega_2(v)|^2 + |\upomega_3(v)|^2$. We see that a Weierstrass representation determines a $\mathbb{T}$-reduction of the frame bundle $F(TM)$ thought of as a principal $\Cx^\times$-bundle (recall the discussion on p.\hspace{2.5pt}\pageref{reductions}). Let us denote the corresponding Riemannian metric $\Re(h)$ on $M$ by $g$.

\begin{lemma}\label{indmetric}
If $\upomega$ comes from a conformal immersion $\upvarphi$, then the induced metric $\upvarphi^* \overline{g}$ equals $\frac{1}{2}g$ (here $\overline{g}$ is the standard Riemannian metric on $\Rl^3$).
\end{lemma}

\begin{proof}
A simple calculation:
$$
g(v) = ||\upomega(v)||^2 = ||d\upvarphi(v)||^2 + ||d\upvarphi(Iv)||^2 = 2\overline{g}(d\upvarphi(v)) = 2 \upvarphi^*\overline{g}(v).
$$
\end{proof}

Let $\upvarphi \colon M \to \Rl^3$ be a conformal immersion. Since the orientation of the normal bundle $TM^{\perp} \subseteq \upvarphi^*(T\Rl^3)$ is chosen, we can think of the second fundamental form $\II$ as a symmetric bilinear form on every tangent space to $M$, i.e. a smooth section of $S^2 T^*M$. Its (2,0)-part is called the \textbf{Hopf differential of} $\boldsymbol{\upvarphi}$ and denoted by $\boldsymbol{\Upomega}$. It is a symmetric $\Cx$-bilinear form on each tangent space, i.e. a smooth section of $S^{2,0} T^*M = S^2 T^{* 1,0}M \subset S^2 T^*_{\Cx}M$. Such sections are called quadratic differentials on $M$. It is a standard fact that the Hopf differential is holomorphic if and only if the immersion is of constant mean curvature (if and only if the Gauss map is harmonic). 

If we start with a Weierstrass representation $(\upomega, \upeta)$ of $M$, there is also a way to get a quadratic differential on $M$. On the one hand, we have $\upomega \colon TM \to T^*\Cx P^1 \cong T^{* 1,0}\Cx P^1$. On the other hand, we may consider the bundle map $TM \cong T^{1,0}M \xrightarrow{d\upeta} T_{\Cx}\Cx P^1 \twoheadrightarrow T^{1,0}\Cx P^1$, where the differential of $\upeta$ is extended $\Cx$-linearly to the complexifications. Let us call this map $d\upeta_{1,0}$. Since $\omega$ and $d\upeta_{1,0}$ are ($\Cx$-linear) bundle maps (covering $\upeta$) to the bundles that are complex-dual to each other, the pairing between the two provides a quadratic differential $\boldsymbol{q} = \bilin{\upomega}{d\upeta_{1,0}}$. Unraveling the definitions\footnote{At the end, the projection onto $T^{1,0}\Cx P^1$ is redundant, for $\upomega(v)$ is already of type (1,0).}, we see that $q(v,w) = \bilin{\upomega(v)}{d\upeta(\frac{w - iIw}{2})_{1,0}} = \frac{1}{2}\bilin{\upomega(v)}{d\upeta(w) - id\upeta(Iw)}$. If $\upeta$ is holomorphic, $q(v,w) = \bilin{\upomega(v)}{d\upeta(w)_{1,0}} = \bilin{\upomega(v)}{d\upeta(w)}$, so $q$ simplifies to $\bilin{\upomega}{d\upeta}$. One may also produce quadratic differentials on $U_{\pm}$ from $\upmu_{\pm}$ and $\upnu_\pm$. Namely, one can take $\upmu_+ \cdot \partial \upnu_+$ over $U_+$ and $\upmu_- \cdot \partial \upnu_-$ over $U_-$, where the dot stands for the symmetric multiplication. We could also write $\upmu_+ \otimes \partial \upnu_+$ and $\upmu_- \otimes \partial \upnu_-$ since $T^{* 1,0}M \otimes_\Cx T^{* 1,0}M = S^2 T^{* 1,0}M$.

\begin{proposition}\label{hopf}
\begin{enumerate}
    \item $\restr{q}{U_+} = \upmu_+ \cdot \partial \upnu_+, \; \, \restr{q}{U_-} = -\upmu_- \cdot \partial \upnu_-$;
    \item If $\upomega$ comes from a conformal immersion, then $q = 2\Upomega$.
\end{enumerate}
\end{proposition}

\begin{proof}
First, we deal with $U_-$. Let $z$ be a local holomorphic coordinate, and write $\upomega = \widetilde{\upomega} dz$ and $\upmu = \widetilde{\upmu} dz$. We have:
$$
\widetilde{\upmu} = \widetilde{\upomega}_1 - i \widetilde{\upomega}_2, \quad \upnu = \frac{\widetilde{\upomega}_3}{\widetilde{\upmu}}, \quad \widetilde{\upomega} = \left( \frac{\widetilde{\upmu}}{2}(1-\upnu^2), \frac{i\widetilde{\upmu}}{2}(1+\upnu^2), \widetilde{\upmu}\upnu \right).
$$
We need to show that $q(\frac{\partial}{\partial x}, \frac{\partial}{\partial x}) = \bilin{\upomega(\frac{\partial}{\partial x})}{d\upeta(\frac{\partial}{\partial z})}$ equals $-\upmu(\frac{\partial}{\partial x}) \partial \upnu(\frac{\partial}{\partial x}) = -\widetilde{\upmu}\upnu_z$, where $\upnu_z$ stands for $\frac{\partial \upnu}{\partial z}$ as usual. If we look at $(\upomega, \upeta)$ as a bundle map to $\mathcal{O}_{[Q]}(-1) \to [Q]$, then, in order to be able to pair $\upomega(\frac{\partial}{\partial x})$ and $d\upeta(\frac{\partial}{\partial z})$, we need an explicit isomorphism $\mathcal{O}_{[Q]}(-1) \cong T^*[Q]$. The trick is to first represent $(\upomega, \upeta)$ as a bundle map to $\mathcal{O}_{\Cx P^1}(-2) \to \Cx P^1$ by means of the bundle isomorphism
$$
\xymatrix{
\mathcal{O}_{\Cx P^1}(-2) \ar[r]_-{\sim}^-{\widetilde{\upgamma}} \ar@{->>}[d] & \mathcal{O}_{[Q]}(-1) \ar@{->>}[d] \\
\Cx P^1 \ar[r]_-{\sim}^-{\widetilde{\uptheta}} & [Q]}
$$
and then to use the isomorphism $\mathcal{O}_{\Cx P^1}(-2) \cong T^*\Cx P^1$ constructed at the end of Section \ref{technical}. Since we work over $U_-$, the map $\widetilde{\uptheta}^{-1} \circ \upeta \colon M \to \Cx P^1$ is $[1:\upnu]$. Given $p \in U_-$, $\upomega(\restr{\frac{\partial}{\partial x}}{p}) = (\upeta(p), (\widetilde{\upomega}_1 (p), \widetilde{\upomega}_2 (p), \widetilde{\upomega}_3 (p))) \in \mathcal{O}_{[Q]}(-1)_{\upeta(p)}$. Let $([1:\upnu(p)], (a,b)) \in \mathcal{O}_{\Cx P^1}(-1)_{[1:\upnu(p)]}$ be any of the two points over it under the quadratic map $\upgamma \colon \mathcal{O}_{\Cx P^1}(-1)_{[1:\upnu(p)]} \twoheadrightarrow \mathcal{O}_{[Q]}(-1)_{\upeta(p)}$. This means that
$$
\begin{cases}
a^2 - b^2 = \widetilde{\upomega}_1 (p) \\
i(a^2 + b^2) = \widetilde{\upomega}_2 (p) \\
2ab = \widetilde{\upomega}_3 (p).
\end{cases}
$$
This also means that $\widetilde{\upgamma}^{-1}(\upomega(\restr{\frac{\partial}{\partial x}}{p})) = ([1:\upnu(p)], (a,b) \otimes (a,b))$. Under the isomorphism $\mathcal{O}_{\Cx P^1}(-2) \cong T^{* 1,0}\Cx P^1$, this vector corresponds to 
$$
-2a^2 d \big(\frac{z_1}{z_0} \big)_{[1:\upnu(p)]} = -(\widetilde{\upomega}_1(p) - i\widetilde{\upomega}_2(p)) d \big(\frac{z_1}{z_0} \big)_{[1:\upnu(p)]} = -\widetilde{\upmu}(p) d \big(\frac{z_1}{z_0} \big)_{[1:\upnu(p)]}.
$$
But then
\begin{align*}
q \big(\frac{\partial}{\partial x}, \frac{\partial}{\partial x} \big) &= \bbilin{\upomega \big(\frac{\partial}{\partial x} \big)}{d\upeta \big(\frac{\partial}{\partial z} \big)} \\
&= \bbilin{-\widetilde{\upmu}(p) d \big(\frac{z_1}{z_0} \big)_{[1:\upnu(p)]}}{d[1:\upnu]_p \big(\restr{\frac{\partial}{\partial z}}{p} \big)} \\
&= -\widetilde{\upmu}(p) d \big(\frac{z_1}{z_0} \circ [1:\upnu] \big)_p \big(\restr{\frac{\partial}{\partial z}}{p} \big) \\
&= -\widetilde{\upmu}(p) d\upnu_p \big(\restr{\frac{\partial}{\partial z}}{p} \big) \\
&= -\widetilde{\upmu}(p)\upnu_z(p) = -\upmu_- \cdot \partial \upnu_- \big( \frac{\partial}{\partial x}, \frac{\partial}{\partial x} \big).
\end{align*}
Similar computation would work over $U_+$, but we can also take a shortcut. First of all, on the intersection $U_+ \cap U_-$,
\begin{multline*}
\upmu_+ \cdot \partial \upnu_+ = \upmu_+ \cdot \partial \big( - \frac{1}{\upnu_-} \big) = \frac{1}{\upnu_-^2} \upmu_+ \cdot \partial \upnu_- = -\frac{\upnu_+ \upmu_+}{\upnu_-} \cdot \partial\upnu_- \\
= -\frac{\upomega_3}{\upnu_-} \cdot \partial\upnu_- = -\upmu_- \cdot \partial\upnu_- = q.
\end{multline*}
Also, $q$ and $\upmu_+ \cdot \partial \upnu_+$ are both zero in the interior of $M_-$, since both $\upeta$ and $\upnu_+$ are constant there. But then these quadratic differentials must agree on the (topological) boundary of $M_-$ as well by continuity, which completes the proof of the first part.

Now, assume that $\upomega$ comes from a conformal immersion $\upvarphi$. We can use the same trick again. When restricted to $\Int(M_-)$, $\upvarphi$ is an immersion to a horizontal plane. The latter is totally geodesic in $\Rl^3$, hence its second fundamental form vanishes, and so does $\restr{\Upomega}{\Int M_-}$. If we show that $q$ and $2\Upomega$ agree on $U_-$, invoking the same continuity argument will finish the proof. As before, let $z$ be a local holomorphic coordinate. Since $\Upomega$ is $\Cx$-bilinear, $\Upomega(\frac{\partial}{\partial x}, \frac{\partial}{\partial x}) = \Upomega(\frac{\partial}{\partial z}, \frac{\partial}{\partial z})$. But $\Upomega$ is the $(2,0)$-part of the second fundamental form, so $\Upomega(\frac{\partial}{\partial z}, \frac{\partial}{\partial z}) = \II(\frac{\partial}{\partial z}, \frac{\partial}{\partial z})$. Let us extend $d\upvarphi$ and the Euclidean Riemannian metric $\overline{g}$ \linebreak $\Cx$-(bi)linearly to the complexifications. Then, given $p$ in the domain of our local coordinate,
$$
\II \big(\frac{\partial}{\partial z}, \frac{\partial}{\partial z} \big)(p) = \II \big(\upvarphi_*\frac{\partial}{\partial z}, \upvarphi_*\frac{\partial}{\partial z} \big)(p) = \big \langle \overline{\nabla}_{\upvarphi_*\frac{\partial}{\partial z}} \big(\upvarphi_*\frac{\partial}{\partial z} \big)(\upvarphi(p)) \mid \upeta(p) \big \rangle = \cross{\upvarphi_{zz}(p)}{\upeta(p)},
$$
where $\upvarphi_* \frac{\partial}{\partial z}$ is a complex vector field on $\upvarphi(M)$ (since we work locally, $\upvarphi$ may be assumed to be an embedding). Therefore, recalling the formula for the Gauss map via $\upnu$ from Proposition \ref{gauss}, we have:
\begin{equation}\label{IIzz}
\II \big(\frac{\partial}{\partial z}, \frac{\partial}{\partial z} \big) = \frac{2X_{zz}\Re\upnu + 2Y_{zz}\Im\upnu + Z_{zz}(|\upnu|^2 - 1)}{|\upnu|^2 + 1}.
\end{equation}
We need to show that this equals $-\frac{\widetilde{\upmu}\upnu_z}{2}$. By definition, $\widetilde{\upomega} = 2\upvarphi_z$, so
$$
X_z = \frac{\widetilde{\upmu}}{4}(1-\upnu^2), \quad Y_z = \frac{i\widetilde{\upmu}}{4}(1+\upnu^2), \quad Z_z = \frac{1}{2}\widetilde{\upmu}\upnu.
$$
Inserting this into \eqref{IIzz} yields:
\begin{align*}
\II &\big(\frac{\partial}{\partial z}, \frac{\partial}{\partial z} \big) = \\[0.5em]
&= \frac{(\widetilde{\upmu}_z(1-\upnu^2) - 2\widetilde{\upmu}\upnu\upnu_z)\Re\upnu + i(\widetilde{\upmu}_z(1+\upnu^2) + 2\widetilde{\upmu}\upnu\upnu_z)\Im\upnu + (\widetilde{\upmu}_z\upnu + \widetilde{\upmu}\upnu_z)(|\upnu|^2-1)}{2(|\upnu|^2 + 1)} \\
&= \frac{\widetilde{\upmu}_z\upnu - \widetilde{\upmu}_z\upnu^2\overline{\upnu} - 2\widetilde{\upmu}\upnu\upnu_z\overline{\upnu} + \widetilde{\upmu}_z\upnu(|\upnu|^2-1) + \widetilde{\upmu}\upnu_z(|\upnu|^2-1)}{2(|\upnu|^2 + 1)} \\
&= \frac{-\widetilde{\upmu}_z\upnu|\upnu|^2 - 2\widetilde{\upmu}\upnu_z|\upnu|^2 + \widetilde{\upmu}_z\upnu|\upnu|^2 + \widetilde{\upmu}\upnu_z(|\upnu|^2-1)}{2(|\upnu|^2 + 1)} \\
&= \frac{-\widetilde{\upmu}\upnu_z(|\upnu|^2 + 1)}{|\upnu|^2 + 1} = -\frac{\widetilde{\upmu}\upnu_z}{2},
\end{align*}
which was to be shown. This completes part 2 and thus the whole proof.
\end{proof}

Since a Weierstrass representation gives a Riemannian metric on $M$, it also gives the corresponding Hodge star operator and codifferential.

\begin{proposition}\label{mean}
If $\upomega$ comes from a conformal immersion, then the function $d^*\upomega$ is real-valued (or rather, $\Rl^3$-valued, as it is actually a triple of functions) and the mean curvature of the immersion can be expressed as $H = -\cross{d^*\upomega}{\upeta}$.
\end{proposition}

Here $\cross{-}{-}$ is the standard dot product on $\Rl^3$ and thus on $\Rl^3$-valued functions.

\begin{proof}
Let $z$ be a local holomorphic coordinate on $M$. The matrix of the shape operator $s$ in the local orthonormal (with respect to the induced metric) frame $\frac{\frac{\partial}{\partial x}}{||d\upvarphi(\frac{\partial}{\partial x})||}, \frac{\frac{\partial}{\partial y}}{||d\upvarphi(\frac{\partial}{\partial y})||}$ coincides with that of $\II$, that is, with
$$
\II = \frac{1}{||d\upvarphi(\frac{\partial}{\partial x})|| \, ||d\upvarphi(\frac{\partial}{\partial y})||} \begin{pmatrix} \cross{\upvarphi_{xx}}{\upeta} & \cross{\upvarphi_{xy}}{\upeta} \\
\cross{\upvarphi_{xy}}{\upeta} & \cross{\upvarphi_{yy}}{\upeta}
\end{pmatrix}.
$$
Let $\upomega = \widetilde{\upomega} dz$, where $\widetilde{\upomega} = \upomega(\frac{\partial}{\partial z}) = \upomega(\frac{\partial}{\partial x}) = 2\upvarphi_z$ is a smooth function $M \to \Cx^3$. Then $h = ||\widetilde{\upomega}||^2 dz \otimes d\overline{z}$ and $g = ||\widetilde{\upomega}||^2 dz d\overline{z} = ||\widetilde{\upomega}||^2 (dx^2 + dy^2)$. It follows that $||d\upvarphi(\frac{\partial}{\partial x})|| = ||d\upvarphi(\frac{\partial}{\partial y})|| = \frac{||\widetilde{\upomega}||}{\sqrt{2}}$. Taking into account that $\frac{\partial^2}{\partial z \partial \overline{z}} = \frac{1}{4}(\frac{\partial^2}{\partial x^2} + \frac{\partial^2}{\partial y^2})$, we arrive at the following formula for the mean curvature:
\begin{equation}\label{meanequation}
H = \frac{\tr (s)}{2} = \frac{1}{2 ||d\upvarphi(\frac{\partial}{\partial x})|| \, ||d\upvarphi(\frac{\partial}{\partial y})||} (\cross{\upvarphi_{xx}}{\upeta} + \cross{\upvarphi_{yy}}{\upeta}) = \frac{4}{||\widetilde{\upomega}||^2} \cross{\upvarphi_{z \overline{z}}}{\upeta}.
\end{equation}
It remains to compute $\upvarphi_{z \overline{z}}$. Observe that the volume form of the $g$ is $\upomega_g = \sqrt{\det g} \, dx \wedge dy = \frac{i ||\widetilde{\upomega}||^2}{2} dz \wedge d\overline{z}$. On the other hand, we have $d\upomega = \overline{\partial} \upomega = -\widetilde{\upomega}_{\overline{z}} \hspace{2pt} dz \wedge d\overline{z}$. This means that $\hast d \upomega = \frac{2i \widetilde{\upomega}_{\overline{z}}}{||\widetilde{\upomega}||^2}$ and thus $d^* \upomega = -\hast d \hast \upomega = -\frac{2\widetilde{\upomega}_{\overline{z}}}{||\widetilde{\upomega}||^2}$. We obtain:
$$
\upvarphi_{z \overline{z}} = \frac{\widetilde{\upomega}_{\overline{z}}}{2} = \frac{- ||\widetilde{\upomega}||^2}{4} d^* \upomega.
$$
Plugging this into \eqref{meanequation} yields the desired formula. This computation also show that
$$
d^* \upomega = -\frac{4\upvarphi_{z \overline{z}}}{||\widetilde{\upomega}||^2} = -\frac{\upvarphi_{xx} + \upvarphi_{yy}}{||\widetilde{\upomega}||^2},
$$
which is $\Rl^3$-valued. This completes the proof.
\end{proof}

This formula for $H$ illustrates the fact that if $\upomega$ is holomorphic, then $\upvarphi$ is minimal.

Finally, we express the Gaussian curvature of a conformal immersion in terms of its Weierstrass representation. To this end, we need a brief digression.

Let $M$ be a smooth manifold, and let $(E,h)$ be a Hermitian vector bundle over $M$. Let $g = \Re(h)$ stand for the corresponding Euclidean bundle metric in $E$. We can extend $g$ $\Cx$-$\frac{3}{2}$-linearly to a Hermitian bundle metric on the complexification $E_\Cx$. The decomposition $E_\Cx = E^{1,0} \oplus E^{0,1}$ becomes orthogonal, and the standard isomorphism $E \isoto E^{1,0}$ is almost unitary: it decreases the Hermitian metric by a factor of 2. Note that the complex conjugation on $E_\Cx$ conjugates the Hermitian metric. Next, if we push $g$ forward along the isomorphism $\hat{g} \colon E \isoto E^*$ (also known as the musical isomorphism), it becomes compatible with the dual almost complex structure $I^*$ in $E^*$. Hence $E^*$ also becomes a Hermitian vector bundle. Nevertheless, $\hat{g}$ is $\Cx$-antilinear and conjugates the Hermitian metrics (although it is isometric with respect to the Euclidean metrics). In the same way as above, we extend the Euclidean metric on $E^*$ to a Hermitian metric on $E^*_\Cx$. We have the following commutative diagram:
$$
\xymatrix{
E \ar[d]_(0.64)*!/u10pt/{\rotatebox{-90}{$\sim$}}^-{\widehat{g}} \ar[dr]_*!/u6pt/{\rotatebox{-40}{$\sim$}}^-{\widehat{h}} \\
E^* \ar[r]_-{\sim} & E^{* 1,0}
}
$$
Note that, according to our conventions, the bottom map \textit{increases} the Hermitian metric by a factor of 2. Finally, we extend all these Hermitian metrics to all the tensor, symmetric, and exterior products and their direct sums in a standard way. 

\begin{example}
We can extend $g$ to $S^2 V^*$ and then extend it $\Cx$-$\frac{3}{2}$-linearly to $S^2 V^*_\Cx$. With respect to this Hermitian metric, the decomposition $S^2 V^*_\Cx = S^{2,0} V^* \oplus S^{1,1} V^* \oplus S^{0,2} V^*$ is orthogonal, and the induced Hermitian metrics on the summands are exactly the ones mentioned in the previous paragraph. The same applies to other symmetric, exterior, and tensor powers.
\end{example}

Now, let $M$ be a Riemann surface and $\upomega$ its Weierstrass representation. Let $z$ be a local holomorphic coordinate, and let $\upomega = \widetilde{\upomega} dz$. Then:

\begin{itemize}
    \item $\frac{1}{||\widetilde{\upomega}||} \frac{\partial}{\partial x}, \frac{1}{||\widetilde{\upomega}||} \frac{\partial}{\partial y}$ is a local Euclidean orthonormal frame for $TM$.
    \item $\frac{\sqrt{2}}{||\widetilde{\upomega}||} \frac{\partial}{\partial z}$ and $\frac{\sqrt{2}}{||\widetilde{\upomega}||} \frac{\partial}{\partial \overline{z}}$ are local Hermitian orthonormal frames for $T^{1,0}M$ and $T^{0,1}M$, respectively.
    \item $||\widetilde{\upomega}||dx, ||\widetilde{\upomega}||dy$ is a local Euclidean orthonormal frame for $T^*M$.
    \item $\frac{||\widetilde{\upomega}||}{\sqrt{2}}dz$ and $\frac{||\widetilde{\upomega}||}{\sqrt{2}}d\overline{z}$ are local Hermitian orthonormal frames for $T^{* 1,0}M$ and $T^{* 0,1}M$, respectively.
    \item $\frac{||\widetilde{\upomega}||^2}{2}dz^2, \frac{||\widetilde{\upomega}||^2}{2} dz d\overline{z}$, and $\frac{||\widetilde{\upomega}||^2}{2} d\overline{z}^2$ are local Hermitian orthonormal frames for $S^{2,0}T^*M$, $S^{1,1}T^*M$, and $S^{0,2}T^*M$, respectively.
    \item $\frac{||\widetilde{\upomega}||^2}{2} dz \wedge d\overline{z}$ is a local Hermitian orthonormal frame for $\extp^{1,1}T^*M$.
\end{itemize}

For example, it follows that $||\upomega|| = \sqrt{2}, ||g|| = ||h|| = 2, ||\upomega_g|| = 1$.

\begin{proposition}\label{gaussian}
If $\upomega$ comes from a conformal immersion, then the Gaussian curvature of the immersion can be expressed as $K = H^2 - 4||\Upomega||^2 = H^2 - ||q||^2$.
\end{proposition}

\begin{proof}
Let $z$ be a local holomorphic coordinate. The Gaussian curvature of $\upvarphi$ is the determinant of the shape operator, and the same argument as in the proof of Proposition \ref{mean} yields:
$$
K = \det(s) = \frac{4}{||\widetilde{\upomega}||^4}(\cross{\upvarphi_{xx}}{\upeta}\cross{\upvarphi_{yy}}{\upeta} - \cross{\upvarphi_{xy}}{\upeta}^2).
$$
It can be easily seen by expressing $\frac{\partial}{\partial x}$ and $\frac{\partial}{\partial y}$ via $\frac{\partial}{\partial z}$ and $\frac{\partial}{\partial \overline{z}}$ that
$$
K = \frac{16}{||\widetilde{\upomega}||^4}(\cross{\upvarphi_{z\overline{z}}}{\upeta}^2 - |\cross{\upvarphi_{zz}}{\upeta}|^2) = H^2 - \frac{16}{||\widetilde{\upomega}||^4}|\cross{\upvarphi_{zz}}{\upeta}|^2.
$$
As we have already seen, $\Upomega = \cross{\upvarphi_{zz}}{\upeta} dz^2$. Consequently, $||\Upomega|| = \frac{2}{||\widetilde{\upomega}||^2}|\cross{\upvarphi_{zz}}{\upeta}|$, and the rest follows.
\end{proof}

Let $\upvarphi \colon M \to \Rl^3$ be a conformal immersion. Recall that a point $p \in M$ is called umbilical (with respect to $\upvarphi$) if the principal curvatures of $\upvarphi$ at this point are equal. We finish this subsection by giving a number of alternative characterizations of umbilical points.

\begin{proposition}\label{umbilical}
Let $M$ be a Riemann surface and $\upvarphi \colon M \to \Rl^3$ a conformal immersion. The following are equivalent for $p \in M$:
\begin{enumerate}[(i)]
    \item $p$ is umbilical.
    \item $K(p) = H^2(p)$.
    \item The Hopf differential $\Upomega$ vanishes at $p$.
    \item The Gauss map $\upeta$ is antiholomorphic at $p$.
\end{enumerate}
If $\upvarphi$ is minimal, then the Gaussian curvature $K$ is nonpositive and vanishes exactly at the umbilical points.
\end{proposition}

Although some of the above is standard, the point is that we can give it a new proof based on the ideas we have developed.

\begin{proof}
To begin with, $(ii) \Leftrightarrow (iii)$ follows from Proposition \ref{gaussian}, and $(iii) \Leftrightarrow (iv)$ can be easily seen by diagonalizing the shape operator. So we need only show that $(i) \Leftrightarrow (ii)$. By Proposition \ref{hopf}(2), the zeroes of $\Upomega$ coincide with those of $q$. Take any $p \in M$. Since $\upomega$ maps $T_pM$ isomorphically onto $T_{\upeta(p)}^*\Cx P^1 \cong T_{\upeta(p)}^{* 1,0}\Cx P^1$, $p$ is a zero of $q$ if and only if $d\upeta$ takes $T_p^{1,0}M$ to $T_{\upeta(p)}^{0,1}\Cx P^1$, which means precisely that $\upeta$ is antiholomorphic at $p$.

If $\upvarphi$ is minimal, then we have $K = -4||\Upomega||^2$ by Proposition \ref{gaussian}, hence the last assertion of the proposition follows. Since the Gauss map of a minimal conformal immersion is holomorphic, the above set can also be described as the set of critical points of the Gauss map in this case.
\end{proof}

\subsection{The problem of periods}

We want to analyze when, how, and to what extent a conformal immersion can be recovered from its Weierstrass representation. First we discuss the general problem of recovering a smooth map to a vector space from its differential.

Let $M$ be a connected smooth manifold, $V$ a finite-dimensional real or complex vector space, and $\upeta$ a closed $V$-valued 1-form on $M$. The goal is to find a smooth function $M \to V$ whose differential is $\upeta$, that is, to check whether $\upeta$ is exact, and if not, try to fix that. It is well known that $\upeta$ is exact if and only if it is conservative, i.e. if its integrals along piecewise smooth paths depend only on the endpoints of the paths. This suggests a way to measure $\upeta$'s nonexactness and a way to fix it. Pick a point $p_0 \in M$. Let $\int_{p_0} \hspace{-0.2em} \upeta \colon \uppi_1(M, p_0) \to V$ stand for the group homomorphism that sends a homotopy class $[\upgamma]$ of loops in $M$ based at $p_0$ to $\int_{\widetilde{\upgamma}} \upeta$, where $\widetilde{\upgamma}$ is a piecewise smooth representative in $[\upgamma]$. Such a representative clearly exists, and the homomorphism is well-defined because $\upeta$ is closed. If $q_0$ is another base point and $\upgamma_{p_0}^{q_0}$ is any piecewise smooth path from $p_0$ to $q_0$, we have an isomorphism
$$
\uppi_1(M, p_0) \isoto \uppi_1(M, q_0), \; [\gamma] \mapsto [\upgamma_{p_0}^{q_0}]^{-1} \cdot [\upgamma] \cdot [\upgamma_{p_0}^{q_0}],
$$
that makes the following diagram commute:
\begin{equation}\label{basepoint}
\xymatrix{
\uppi_1(M, p_0) \ar[rr]^-{\sim} \ar[dr]_{\int_{p_0} \hspace{-0.2em} \upeta} && \uppi_1(M, q_0) \ar[dl]^{\int_{q_0} \hspace{-0.2em} \upeta} \\
& V}
\end{equation}
Hence, the choice of a base point is not important. Next, $\upeta$ is exact if and only if $\int_{p_0} \hspace{-0.2em} \upeta$ is trivial, i.e. if the subgroup $\Upgamma = \ker\left(\int_{p_0} \hspace{-0.2em} \upeta \right) \trianglelefteq \, \uppi_1(M, p_0)$ equals the whole fundamental group. Hence, the quotient group $\uppi_1(M, p_0)/\Upgamma$ measures how badly $\upeta$ fails to be exact. It is isomorphic to $\Im(\int_{p_0} \hspace{-0.2em} \upeta) \subseteq V$ and hence is abelian (which is equivalent to $[\uppi_1(M, p_0), \uppi_1(M, p_0)] \subseteq \Upgamma$). It follows from the commutativity of \eqref{basepoint} that the image of $\int_{p_0} \hspace{-0.2em} \upeta$ does not depend on the choice of a base point. This image is called the \textbf{group of periods of} $\boldsymbol{\upeta}$ and its elements are called the \textbf{periods of} $\boldsymbol{\upeta}$. 

If $\upeta$ is nonexact -- that is, if it has nonzero periods, -- one possible way to resolve this issue would be to pass to some covering $\uppi \colon E \to M$ such that the pullback of $\upeta$ by $\uppi$ becomes exact (that is, such that $[\upeta] \in \Ker(\uppi^* \colon H^1_{dR}(M) \to H^1_{dR}(E))$). It would also be useful to find the ``smallest'' covering space with such property. Consider the category of such coverings of $M$ with arbitrary covering homomorphisms as morphisms. Let us investigate when a covering $\uppi \colon E \to M$ lies in this category. Pick any $p_0 \in M$ and any $\widetilde{p}_0 \in E$ over it. This covering is an object of our category if and only if $\int_{\upgamma} \uppi^*\upeta = 0$ for each piecewise smooth loop in $E$ based at $\widetilde{p}_0$. But this integral equals $\int_{\uppi \circ \upgamma} \upeta$, so its vanishing is equivalent to $\uppi_*[\upgamma] \in \Upgamma$, where $\uppi_* \colon \uppi_1(E,\widetilde{p}_0) \to \uppi_1(M,p_0)$. Consequently, $(E,\uppi)$ is an object of our category if and only if $\uppi_* \uppi_1(E, \widetilde{p}_0) \subseteq \Upgamma$. Since $\Upgamma$ is normal, this condition does not depend on the choice of $\widetilde{p}_0$ over $p_0$, and it is also independent of the choice of $p_0$. Now, a plausible candidate for the ``smallest'' covering space where the pullback of $\upeta$ becomes exact would be the terminal object of the category considered (so we get uniqueness up to isomorphism for free). The theory of covering spaces tells us that, under the correspondence between isomorphism classes of coverings of $M$ and conjugacy classes of subgroups of $\pi_1(M, p_0)$, this ``smallest'' covering corresponds to $\Upgamma$. Let us denote it by $\boldsymbol{\uppi_{\upeta}} \colon \boldsymbol{M_{\upeta}} \to M$. It satisfies the following two properties:

\begin{itemize}
    \item It is normal.
    \item Its groups of deck transformations is abelian. Since the automorphism group of a normal covering is isomorphic to the quotient of the fundamental group of the base by the induced subgroup (in our case, $\Aut(\uppi_\upeta)$ is isomorphic to the group of periods), it is the same as to say that the induced subgroup contains the commutator subgroup of the fundamental group of the base.
\end{itemize}

Such covering maps are called \textbf{abelian}. Since $\uppi_{\upeta}^*\upeta$ is exact, one of the functions $M_{\upeta} \to V$ whose differential is $\uppi_{\upeta}^*\upeta$ can be given by the formula $\upvarphi(p) = \int_{\widetilde{p}_0}^p \uppi_{\upeta}^*\upeta$, where the integral is taken along some (any) piecewise smooth path from $\widetilde{p}_0$ to $p$. By construction, $\upvarphi$ is well defined. Any other such function differs from $\upvarphi$ by a translation in $V$. Also, note that, by construction, $\upvarphi$ is equivariant with respect to the actions of $\uppi_1(M, p_0)$ on $M_\upeta$ by deck transformations and on $V$ by translations by means of $\int_{p_0} \hspace{-0.2em} \upeta$. Finally, unless $\uppi_\upeta$ is the trivial covering (that is, unless $\upeta$ is exact), it is countably-sheeted, for its group of periods (and hence the index of $\Upgamma$) is countable.

\subsection{Recovering an immersion from its Weierstrass representation}

Let $M$ be a connected Riemann surface, and let $\upomega$ be a Weierstrass representation of $M$. If it comes from some conformal immersion, then $\Re(\upomega)$ is the differential of this immersion and hence is closed. As long as this is satisfied, the only remaining obstacle to get a conformal immersion is the problem of periods. We obtain the following

\begin{independentcorollary}\label{recover}
Let $\upomega$ be a Weierstrass representation of $M$ with $\Re(\upomega)$ closed. Then there exists a conformal immersion $M_{\Re(\upomega)} \to \Rl^3$ whose corresponding Weierstrass representation is $\uppi_{\Re(\upomega)}^* \upomega$. All such immersions differ by a translation in $\Rl^3$. Finally, $\upomega$ comes from a conformal immersion of $M$ itself (i.e. $M_{\Re(\upomega)} = M$) if and only if all the periods of $\upomega$ are purely imaginary.
\end{independentcorollary}

The fact that the above map $M_{\Re(\upomega)} \to \Rl^3$ will be a conformal immersion follows from Lemma \ref{conformal}(2).

The condition $\Re(d\upomega) = 0$ on a Weierstrass representation $\upomega$ is called the \textbf{integrability condition}. The condition on the periods of $\upomega$ to be purely imaginary is called the \textbf{period condition}. In these terms, a Weierstrass representation of $M$ comes from a conformal immersion of $M$ if and only if it satisfies both the integrability and period conditions. Checking that the integrability condition is satisfied may prove tedious, especially if a Weierstrass representation is given as a bundle map to, say, $T^*\mathbb{S}^2 \to \mathbb{S}^2$ (because one first needs to represent it as a bundle map to $\mathcal{O}_{[Q]}(-1) \to [Q]$ to obtain the corresponding (1,0)-form). On the other hand, one can take any closed $\Rl^3$-valued 1-form $\upeta$ and produce a $\Cx^3$-valued (1,0)-form $\upeta + i\hast\upeta$. But then one will still need to check that $\sum_{i=1}^3{(\upeta_i + i\hast \upeta_i)^2} = 0$. This problem does not arise in the minimal/holomorphic case because the real part of a holomorphic 1-form is automatically closed, so \textit{any holomorphic Weierstrass representation satisfies the integrability condition}. Hence we obtain

\begin{independentcorollary}\label{recoverholomorphic}
Let $\upomega$ be a holomorphic Weierstrass representation of $M$. Then there exists a minimal conformal immersion $M_{\Re(\upomega)} \to \Rl^3$ whose corresponding Weierstrass representation is $\uppi_{\Re(\upomega)}^* \upomega$. All such immersions differ by a translation in $\Rl^3$. Finally, $\upomega$ comes from a minimal conformal immersion of $M$ itself (i.e. $M_{\Re(\upomega)} = M$) if and only if all the periods of $\upomega$ are purely imaginary.
\end{independentcorollary}

\begin{remark}
\begin{enumerate}
    \item The conformal minimal immersion $\upvarphi \colon M_{\Re(\upomega)} \to \Rl^3$ in Corollary \ref{recoverholomorphic} can be expressed as
    $$
    \upvarphi(\wt{p}) = \Re \int_{\widetilde{p}_0}^{\wt{p}} \left( \frac{1-\widetilde{\upnu}^2}{2}, \frac{i(1+\widetilde{\upnu}^2)}{2}, \widetilde{\upnu} \right) \widetilde{\upmu},
    $$
    where $\widetilde{\upmu} = \uppi_{\Re(\upomega)}^* \upmu$ and $\widetilde{\upnu} = \uppi_{\Re(\upomega)}^* \upnu$ are a holomorphic 1-form and a meromorphic function on $M_\upeta$, respectively, $\widetilde{p}_0 \in M_{\Re(\upomega)}$ is any, and the integral is taken along any piecewise smooth curve from $\widetilde{p}_0$ to $\wt{p}$.
    \item If $M$ is simply connected (hence biholomorphic to $\Cx$, the unit disc, or $\Cx P^1$), then any Weierstrass representation of $M$ satisfies the period condition. Hence, if it satisfies the integrability condition, it gives a conformal immersion of $M$ to $\Rl^3$. For instance, any holomorphic Weierstrass representation of $M$ arises from a minimal conformal immersion $M \to \Rl^3$. Observe that if $M$ is $\Cx P^1$, a Weierstrass representation of $M$ cannot be holomorphic, since otherwise the immersion would be constant by the maximum principal for harmonic functions. More generally, a compact Riemann surface does not admit holomorphic Weierstrass representations satisfying the period condition.
    \item The above implies that a Weierstrass representation satisfies the integrability condition if and only if it comes from a conformal immersion locally. It then comes from a conformal immersion over any simply connected open $U \subseteq M$.
\end{enumerate}
\end{remark}

At the end of this section and before we proceed to spinor representations, we briefly discuss the notion of adjoint surfaces. It is easy to see that the coordinate definition from Section \ref{classical} translates into the following. Let $\upvarphi$ be a conformal immersion of a connected Riemann surface $M$ to $\Rl^3$. A smooth map $\hat{\upvarphi} \colon M \to \Rl^3$ is called \textbf{adjoint} to $\upvarphi$ if $d\hat{\upvarphi} = \hast d\upvarphi$. In terms of the Weierstrass representation $\upomega$ of $\upvarphi$, this means $d\hat{\upvarphi} = \Im(\upomega)$. If we write $\hat{\upomega} = 2\partial \hat{\upvarphi}$, yet another way to restate the adjointness condition would be to say that $\hat{\upomega} = \hast \upomega = -I^* \upomega = -i \upomega$. We see that if such $\hat{\upvarphi}$ exists, then necessarily:
\begin{enumerate}
    \item $\hat{\upvarphi}$ is a conformal immersion (by Lemma \ref{conformal}).
    \item $\upomega$ is closed and hence holomorphic, so $\upvarphi$ and $\hat{\upvarphi}$ are minimal.
\end{enumerate}
So the notion of adjointness only makes sense for minimal conformal immersions. Perhaps, a more general way to look at it would be to start with a Weierstrass representation.

Let $\upomega$ be a holomorphic Weierstrass representation of $M$. It gives minimal conformal immersions $M_{\Re(\upomega)} \to \Rl^3$ and $M_{\Im(\upomega)} \to \Rl^3$. If we want them to be defined on the same space and thus give a pair of adjoint minimal immersions, we need to pass to a ``larger'' covering space $M_{\upomega}$. We have $\upvarphi, \hat{\upvarphi} \colon M_{\upomega} \to \Rl^3$ with $d\upvarphi = \Re(\uppi^*_{\upomega} \upomega), \, d\hat{\upvarphi} = \Im(\uppi^*_{\upomega} \upomega)$. Therefore, $f = \upvarphi + i\hat{\upvarphi} \colon M_{\upomega} \to \Cx^3$ is a holomorphic isotropic immersion, since its differential $\uppi^*_{\upomega} \upomega$ is a (1,0)-form. This $f$ can be defined on $M$ itself if and only if all the periods of $\upomega$ are zero, i.e. if $\upomega$ is exact. 

\section{Spin structures and spinor representations}\label{spin_section}

The interplay between conformal immersions and Weierstrass representations can be studied in terms of spinors. The case of surfaces is rather special in spin geometry. The reason for this is that in dimension $n \geqslant 3$ the fundamental group of $\SO(n)$ is $\mathbb{Z}/2\mathbb{Z}$, so the two-sheeted covering map $\Spin_n \twoheadrightarrow \SO(n)$ is universal and $\Spin_n$ is simply connected. In dimension 2, on the other hand, $\SO(2)$ is a circle, so $\Spin_2$ is also a circle, and the covering map $\Spin_2 \twoheadrightarrow \SO(2)$ is just a circle winding twice around itself. This bug of dimension 2 turns out to be a feature, for it allows for a very elegant and practical reformulation of the notion of a spin structure as a so-called square root of the complex line bundle associated with a principal $\SO(2)$-bundle.

\subsection{Spin structures as square roots}

We begin with a detailed discussion of the ``square root'' reformulation and prove its equivalence to the original notion of spin structure. In fact, we will formulate and prove a rather general result about principal $\SO(2)$-bundles and their spin structures over arbitrary smooth manifolds and then get the desired fact for surfaces as a special case. We will make use of category theory, mostly to show that the equivalences between various notions of spin structures are functorial, but also because it makes both the reformulations and the proofs quite elegant.

Let $M$ be a smooth manifold and let $Q$ be a principal $\SO(2)$-bundle over $M$. The standard definition of a spin structure is a $\Spin_2$-reduction $P \twoheadrightarrow Q$ of $Q$. Two spin structures $P \twoheadrightarrow Q$ and $P' \twoheadrightarrow Q$ are considered equivalent if they are equivalent as $\Spin_2$-reductions, i.e. if there is an isomorphism $P \isoto P'$ of principal $\Spin_2$-bundles making the following diagram commutative:
$$
\xymatrix{
P \ar[rr]_{\sim} \ar@{->>}[dr] && P' \ar@{->>}[dl] \\
& Q}
$$
So spin structures on $Q$ together with their equivalences form a category (whose morphisms are all isomorphisms) denoted by $\coslice{\Prin(M, \Spin_2)}{}{Q}$ (for the reasons explained below). 

Let $L = Q \times_{\SO(2)} \Rl^2$ be the Hermitian line bundle over $M$ associated to $Q$.

\begin{definition}
A \textbf{Hermitian square root of} $\boldsymbol{L}$ is a Hermitian line bundle $N$ over $M$ together with a smooth bundle map $N \twoheadrightarrow L$ that is complex-quadratic on each fiber and that sends unit vectors to unit vectors (in other words, it is length-squaring). Two Hermitian square roots $N$ and $N'$ of $L$ are called \textbf{equivalent} if there is a smooth unitary isomorphism $N \isoto N'$ making the diagram
$$
\xymatrix{
N \ar[rr]_-{\sim} \ar@{->>}[dr] && N' \ar@{->>}[dl] \\
& L}
$$
commutative. Hermitian square roots of $L$ and their equivalences form a category denoted by $\boldsymbol{\sqrt[H]{L}}$.
\end{definition}

Now let us forget about the Hermitian metric on $L$ and think of it as an arbitrary smooth complex line bundle over $M$.

\begin{definition}
A (\textbf{smooth}) \textbf{square root of} $\boldsymbol{L}$ is a smooth complex line bundle $N$ over $M$ together with a smooth bundle map $N \twoheadrightarrow L$ that is complex-quadratic on each fiber. Two square roots $N$ and $N'$ of $L$ are called \textbf{equivalent} if there is a smooth isomorphism $N \isoto N'$ of complex line bundles respecting the quadratic maps onto $L$. Square roots of $L$ and their equivalences form a category denoted by $\boldsymbol{\sqrt{L}}$.
\end{definition}

Finally, assume that $M$ is a complex manifold and $L$ is not just smooth but a holomorphic line bundle over $M$.

\begin{definition}
A \textbf{holomorphic square root of} $\boldsymbol{L}$ is a holomorphic complex line bundle $N$ over $M$ together with a holomorphic bundle map $N \twoheadrightarrow L$ that is complex-quadratic on each fiber. Two holomorphic square roots $N$ and $N'$ of $L$ are called \textbf{equivalent} if there is a holomorphic isomorphism $N \isoto N'$ of complex line bundles respecting the quadratic maps onto $L$. Holomorphic square roots of $L$ and their equivalences form a category denoted by $\boldsymbol{\sqrt[\hol]{L}}$.
\end{definition}

We will prove

\begin{proposition}\label{spinsquare}
Let $M$ be a smooth manifold, $Q$ a principal $\SO(2)$-bundle over $M$, and $L$ its associated Hermitian line bundle. Then the categories $\Prin(M, \Spin_2)/Q, \sqrt[H]{L}$, and $\sqrt{L}$ are equivalent. If $M$ is a complex manifold and $L$ is given a holomorphic structure, then these categories are also equivalent to $\sqrt[\hol]{L}$.
\end{proposition}

We will see that there are natural functors between these categories and show that they are equivalences. Before we get started, we need to establish some notation. Let $M$ be a smooth manifold and $G$ a Lie group.

\begin{itemize}
    \item We write $\Prin(M,G)$ for the category of principal $G$-bundles over $M$ with isomorphisms of $G$-bundles as morphisms. If $M$ is a complex manifold and $G$ is a complex Lie group, we write $\Prin^{\hol}(M,G)$ for the category of holomorphic principal $G$-bundles over $M$ and their (holomorphic) isomorphisms.
    \item If $F$ is a smooth manifold and $\uptheta \colon G \times F \to F$ is a smooth action, then $\Fib(M,G,F, \uptheta)$ will stand for the category of fiber bundles over $M$ with model fiber $F$ and structure group $G$ ($G$-bundles with fiber $F$ for brevity) and their isomorphisms as $G$-bundles (we will write $\Fib(M,G,F)$ if there is no ambiguity). If $M$ and $F$ are complex manifolds, $G$ is a complex Lie group, and the action is holomorphic, then we write $\Fib^{\hol}(M,G,F,\uptheta)$ (or simply $\Fib^{\hol}(M,G,F)$) for the category of holomorphic $G$-bundles with fiber $F$ and their isomorphisms as $G$-bundles.
    \item A special case of the above is when $G$ acts linearly on a vector space. In other words, let $(V, \uprho)$ be a (real or complex) representation of $G$. Then any $G$-bundle over $M$ with a fiber $V$ is naturally a vector bundle, and $G$-isomorphisms of these are in particular vector bundle isomorphisms, so we in this case we use a more suggestive notation $\mathrm{VB}_{M, V}^{G, \uprho} = \Fib(M,G,V,\uprho)$. If the representation $\uprho$ is clear from the context, we omit it from the notation. For example, if $G \subseteq \mathrm{GL}(n, \mathbb{K})$ is a closed Lie subgroup, we just write $\mathrm{VB}_{M, \mathbb{K}^n}^{G}$. If $G$ is the entire general linear group, we write $\mathrm{VB}_{M, \mathbb{K}^n}^{\times}$ (this is just the category of rank-$n$ $\mathbb{K}$-vector bundles over $M$ and their isomorphisms). In the holomorphic case, we add the symbol $\hol$ (as if the notation was not already overloaded).
\end{itemize}

All morphisms in the categories above are isomorphisms. There is the associated bundle functor $\Prin(M,G) \to \Fib(M,G,F,\uptheta), P \mapsto P \times_{\uptheta} F$ (or just $P \times_G F$ if the action is clear from the context), which is an equivalence provided that the action $G \curvearrowright F$ is effective. The same remains true in the holomorphic setting. If $G = \SO(n)$, one readily sees that $\mathrm{VB}_{M, \Rl^n}^{\SO(n)}$ can be described as the category of oriented real vector bundles of rank-$n$ with fixed bundle metrics, where morphisms are orientation-preserving isometric isomorphisms. Analogous remarks can be made for the groups $O(n), U(n)$, etc. and in the holomorphic setting.

Now we make a pivotal, albeit trivial, observation. Let $\uppsi \colon G \to H$ be a morphism of Lie groups, and let $Q$ be a principal $H$-bundle over $M$. Then, just as in the case of $\Spin_2$ and $\SO(2)$, $G$-reductions of $Q$ together with their equivalences form a category. Note that $\uppsi$ gives a smooth left action of $G$ on $H$, so the association functor in this case looks like $\widehat{\uppsi} \colon \Prin(M,G) \to \Prin(M,H), P \mapsto P \times_G H$. Now, a $G$-reduction $P \to Q$ is the same as an isomorphism of principal $H$-bundles $P \times_G H \isoto Q$. An isomorphism of $G$-bundles $P \isoto P'$ is an equivalence of $G$-reductions of $Q$ if and only if the induced isomorphism of $H$-bundles $P \times_G H \isoto P' \times_G H$ makes the following diagram commute:
$$
\xymatrix{
P \times_G H \ar[rr]_{\sim} \ar[dr]_(0.57)*!/u5pt/{\rotatebox{-35}{$\sim$}} && P' \times_G H \ar[dl]^(0.57)*!/u5pt/{\rotatebox{35}{$\sim$}} \\
& Q}
$$
This means that the category of $G$-reductions of $Q$ is nothing but the so-called coslice (or comma) category $\coslice{\Prin(M,G)}{}{Q}$, i.e. the category of objects of $\Prin(M,G)$ over $Q$ by means of the association functor $\widehat{\uppsi}$. When we want to specify the functor, we write $\Coslice{\Prin(M,G)}{\widehat{\uppsi}}{Q}$.

We will make use of the following simple statement about coslice categories:

\begin{lemma}\label{coslice}
Consider a commutative square of categories and functors between them:
$$
\xymatrix{
\mathscr{A} \ar[r]^E \ar[d]^F & \mathscr{B} \ar[d]^G \\
\mathscr{C} \ar[r]^H & \mathscr{D}}
$$
Here by ``commutative'' we mean that a particular isomorphism of functors $\upvarepsilon \colon G \circ E \Rightarrow H \circ F$ is fixed. Let $Y$ be in object of $\mathscr{C}$. Then there is functor $(E,\upvarepsilon) \colon \coslice{\mathscr{A}}{F}{Y} \to \coslice{\mathscr{B}}{G}{H(Y)}$ that sends $(X, F(X) \xrightarrow{f} Y)$ to $(E(X), \, G(E(X)) \xrightarrow{H(f) \, \circ \, \upvarepsilon_X} H(Y))$ (and is just $E$ on morphisms). If $E$ is faithful, then so is $(E,\upvarepsilon)$, and if $E$ and $H$ are equivalences, then so is $(E,\upvarepsilon)$.
\end{lemma}

\begin{proof}
Trivial.
\end{proof}

Let us go back to our situation. If we denote the standard basis for $\Rl^2$ as $e_1, e_2$, the Clifford algebra $\mathrm{Cl}_2$ is generated by these two elements with the multiplicative relations $e_1^2 = e_2^2 = -1$ and $e_1 e_2 = -e_2 e_1$. Consequently, $1, e_1, e_2, e_1 e_2$ is the vector space basis for $\mathrm{Cl}_2$. It is easy to see that
$$
\Spin_2 = \set{\cos\hspace{-0.07em}\upalpha + \sin\hspace{-0.12em}\upalpha \hspace{0.15em} e_1 e_2 \mid \upalpha \in [0,2\uppi)},
$$
and the two-sheeted covering map $\Spin_2 \twoheadrightarrow \SO(2)$ sends $\cos\hspace{-0.07em}\upalpha + \sin\hspace{-0.12em}\upalpha \hspace{0.15em} e_1 e_2$ to $\bigl( \begin{smallmatrix}\cos 2\upalpha & -\sin 2\upalpha \\ \sin 2\upalpha & \; \cos 2\upalpha \end{smallmatrix}\bigr)$, i.e. the counterclockwise rotation through the angle $2\upalpha$. Under the identification $\SO(2) \cong \mathbb{T}$, there is a unique isomorphism of Lie groups $\Spin_2 \isoto \mathbb{T}$ making the covering map look like $\mathbb{T} \twoheadrightarrow \mathbb{T}, z \mapsto z^2$. One can easily see that this isomorphism is just
$$
\cos\hspace{-0.09em}\upalpha + \sin\hspace{-0.12em}\upalpha \hspace{0.15em} e_1 e_2 \leftrightarrow e^{i\upalpha}.
$$
We have a diagram of Lie groups:
\begin{equation}\label{diagram1}
\xymatrix{
\Spin_2 \ar@{->>}[d] \ar[r]_(0.59){\sim} & \mathbb{T} \; \ar@{->>}[d] \ar@{^{(}->}[r] & \Cx^{\times} \ar@{->>}[d] \\
\SO(2) \ar[r]_(0.58){\sim} & \mathbb{T} \; \ar@{^{(}->}[r] & \Cx^{\times}}
\end{equation}
where both the central and the right vertical arrows are $z \mapsto z^2$. This gives a one-dimensional (hence irreducible) complex representation of $\Spin_2$.

\begin{observation}\label{spinobs}
$\Spin_2$ is isomorphic to $\mathbb{T}$, hence its irreducible complex representations are classified by the integers: $z \mapsto z^n, n \in \mathbb{Z}$, in terms of the circle. Only two of these representations are faithful: the ones corresponding to $n = \pm 1$. At the same time, the complex spin representation $\Updelta_2^{\Cx}$ of $\Spin_2$ is faithful and is the sum of two non-isomorphic irreducible representations (see, e.g., \cite{lawson-michelsohn}). Thus, they must be the ones corresponding to $\pm 1$. In particular, they are conjugate to each other. Since the orientation of $\Rl^2$ is chosen, the summands of $\Updelta_2^{\Cx}$ are distinguished by the action of the complex volume element $\upomega_{\Cx} = ie_1 e_2 \in \Cx\ell_2$. The one where it acts trivially is denoted by $\Updelta_2^{\Cx \, +}$, while the other (where it acts by $-1$) is $\Updelta_2^{\Cx \, -}$. In this regard, the representation we constructed above (the one with $n = +1$) is $\Updelta_2^{\Cx \, -}$, since $e_1 e_2$ goes to $i$, and its conjugate is $\Updelta_2^{\Cx \, +}$.
\end{observation}

Let $M$ be a smooth manifold. We draw a diagram of categories of principal and vector bundles over $M$ and functors between them:
$$
\resizebox{0.98\hsize}{!}{
\xymatrix@-0.8em{
&&& \mathrm{VB}^{\mathbb{T}}_{M, \Cx} \ar@{^{(}->}[rr] \ar[dd]|\hole && \mathrm{VB}^{\times}_{M, \Cx} \ar[dd]|\hole && \mathrm{VB}^{\times, \hol}_{M, \Cx} \ar[ll] \ar[dd] \\
\Prin(M, \Spin_2) \ar[rr]_-{\sim} \ar[dd] && \Prin(M, \mathbb{T}) \ar@{^{(}->}[rr] \ar[dd] \ar[ur]^(0.51)*!/d5pt/{\rotatebox{28}{$\sim$}} && \Prin(M, \Cx^{\times}) \ar[dd] \ar[ur]^(0.50)*!/d5pt/{\rotatebox{28}{$\sim$}} && \Prin^{\hol}(M, \Cx^{\times}) \ar[ll] \ar[dd] \ar[ur]^(0.52)*!/d5pt/{\rotatebox{27}{$\sim$}} \\
&&& \mathrm{VB}^{\mathbb{T}}_{M, \Cx} \ar@{^{(}->}[rr]|\hole && \mathrm{VB}^{\times}_{M, \Cx} && \mathrm{VB}^{\times, \hol}_{M, \Cx} \ar[ll]|(0.506)\hole \\
\Prin(M, \SO(2)) \ar[rr]_-{\sim} && \Prin(M, \mathbb{T}) \ar@{^{(}->}[rr] \ar[ur]^(0.51)*!/d5pt/{\rotatebox{28}{$\sim$}} && \Prin(M, \Cx^{\times}) \ar[ur]^(0.50)*!/d5pt/{\rotatebox{28}{$\sim$}} && \Prin^{\hol}(M, \Cx^{\times}) \ar[ll] \ar[ur]^(0.52)*!/d5pt/{\rotatebox{27}{$\sim$}} }
}
$$

We will refer to this as the big diagram. The front side of the diagram comes from \eqref{diagram1} (and hence the front squares are commutative). Arrows $\isoto$ coming from the front to the rear side of the diagram are the associated-bundle functors. The right-hand side square of the diagram makes sense only when $M$ is a complex manifold. The vertical arrows on the rear side of the diagram are tensor square functors $L \mapsto L \otimes L$, which we will denote by $T$.

\begin{remark}
To distinguish between various maps and functors, let us denote the identity maps $\mathbb{T} \isoto \mathbb{T}, \, \Cx^{\times} \isoto \Cx^{\times}$ and the embedding $\mathbb{T} \hookrightarrow \Cx^{\times}$ by $\uprho_1$, and the square maps $\mathbb{T} \twoheadrightarrow \mathbb{T}, \, \Cx^{\times} \twoheadrightarrow \Cx^{\times}$, and $\mathbb{T} \to \Cx^{\times}$, all given by $z \mapsto z^2$, by $\uprho_2$.
\end{remark}

\begin{remark}
Note that when $G$ is either $\mathbb{T}$ or $\Cx^{\times}$, tensor squaring a line $G$-bundle is a functor from $\mathrm{VB}^{G, \uprho_1}_{M, \Cx}$ to $\mathrm{VB}^{G, \widetilde{\uprho}_1}_{M, \Cx \hspace{0.06em} \otimes \hspace{0.06em} \Cx}$, where $\widetilde{\uprho}_1$ is the representation of $G$ in $\Cx \otimes \Cx$ induced by $\uprho_1$. But under the isomorphism $\Cx \otimes \Cx \cong \Cx$, it is nothing but $\uprho_2$. So we have\footnote{Here the last arrow is rather tautological. If we have a fiber bundle with fiber $\Cx$ and $G$-valued cocycles, where $G$ acts on $\Cx$ via $\uprho_2$, we can compose those cocycles with $\uprho_2 \colon G \twoheadrightarrow G$ (that is, square them) and think of them as still $G$-valued, but now where $G$ acts on $\Cx$ via $\uprho_1$.} $\mathrm{VB}^{G, \uprho_1}_{M, \Cx} \to \mathrm{VB}^{G, \widetilde{\uprho}_1}_{M, \Cx \hspace{0.06em} \otimes \hspace{0.06em} \Cx} \cong \mathrm{VB}^{G, \uprho_2}_{M, \Cx} \to \mathrm{VB}^{G, \uprho_1}_{M, \Cx}$ (the holomorphic case is analogous). This is what we call $T$. One can easily see that if $G = \Cx^{\times}$, this is just taking a tensor square of a line bundle, whereas if $G = \mathbb{T}$, this is taking a tensor square of a Hermitian line bundle and endowing it with an obvious Hermitian metric.
\end{remark}

The horizontal arrows coming to the central square (the one with $\Prin(M, \Cx^{\times})$ and $\mathrm{VB}^{\times}_{M, \Cx}$~) are faithful\footnote{In fact, the arrows $\hookrightarrow$ are bijective on the isomorphism classes because $\mathbb{T}$ is the maximal compact subgroup of $\Cx^{\times}$ (which is why we draw them as embeddings), but we do not need this.}. The only squares of the diagram where commutativity may not be straightforward are the (left and right) side squares of the cubes.

\begin{lemma}\label{commutative}
The are natural isomorphism of functors making the side squares of the cubes on the big diagram above commutative.
\end{lemma}

\begin{proof}
We deal with the central of the three squares, namely with
$$
\xymatrix{
\Prin(M, \Cx^{\times}) \ar[r]_-{\sim} \ar[d]^{\widehat{\uprho}_2} & \mathrm{VB}^{\times}_{M, \Cx} \ar[d]^T \\
\Prin(M, \Cx^{\times}) \ar[r]_-{\sim} & \mathrm{VB}^{\times}_{M, \Cx} }.
$$
Let $P$ be a principal $\Cx^{\times}$-bundle over $M$. If we send $P$ down and then right, it will go to $(P \times_{\uprho_2} \Cx^{\times}) \times_{\uprho_1} \Cx \cong P \times_{\uprho_2} \Cx$. On the other hand, if we send it first right and then down, it will go to $(P \times_{\uprho_1} \Cx) \otimes (P \times_{\uprho_1} \Cx) \cong P \times_{\widetilde{\uprho}_1} (\Cx \otimes \Cx) \cong P \times_{\uprho_2} \Cx$. The proofs for the other two squares are completely analogous.
\end{proof}

Now we can finally proceed to the

\begin{proof}[Proof of Proposition \ref{spinsquare}]

Let $Q$ be a principal $\SO(2)$-bundle over $M$. We will think of it as a principal $\mathbb{T}$-bundle and, by Lemma \ref{coslice} and the commutativity of the leftmost square on the big diagram, of its spin structures $\coslice{\Prin(M, \Spin_2)}{}{Q}$ as of reductions along the morphism $\uprho_2 \colon \mathbb{T} \to \mathbb{T}$, i.e. as of the category $\Coslice{\Prin(M, \mathbb{T})}{\widehat{\uprho}_2}{Q}$. Let $L = Q \times_{\uprho_1} \Cx$ be the associated Hermitian line bundle. Then, again by Lemma \ref{coslice} and this time by the commutativity of the left side square of the left cube, $\Coslice{\Prin(M, \mathbb{T})}{\widehat{\uprho}_2}{Q}$ is equivalent to $\Coslice{\mathrm{VB}^{\mathbb{T}}_{M, \Cx}}{T}{L}$. An object of the latter category is a Hermitian line bundle $S$ together with a unitary isomorphism $S \otimes S \isoto L$. Here is the key observation (a special case of which already popped up in Section \ref{technical}): such a unitary isomorphism is the same as a bundle map $S \twoheadrightarrow L$ that is complex-quadratic and length-squaring on each fiber. Just as the category of principal bundle reductions can be described as $\coslice{\Prin(M,G)}{}{Q}$, the category $\sqrt[H]{L}$ of Hermitian square roots of $L$ is simply $\Coslice{\mathrm{VB}^{\mathbb{T}}_{M, \Cx}}{T}{L}$.

\begin{remark}
Let $P$ be a spin structure on $Q$. It follows from the observation on page \pageref{spinobs} that the line bundle $S = P \times_{\uprho_1} \Cx$ that $P$ goes to under the equivalence $\Coslice{\Prin(M, \mathbb{T})}{\widehat{\uprho}_2}{Q} \isoto \Coslice{\mathrm{VB}^{\mathbb{T}}_{M, \Cx}}{T}{L}$ is the negative part of the spinor bundle corresponding to $P$. The positive summand is $\overline{S}$, and the entire spinor bundle is $\overline{S} \oplus S$.
\end{remark}

Next, let $Q' = Q \times_{\uprho_1} \Cx^{\times}$ and let $L$ be its associated line bundle. We use the same letter $L$ because this is just the line bundle above without its Hermitian metric (this follows from the commutativity of the top/bottom squares). If we move to the right on the diagram by means of the arrows $\hookrightarrow$ and apply Lemma \ref{coslice} one more time, this will give us a functor from $\Coslice{\Prin(M, \mathbb{T})}{\widehat{\uprho}_2}{Q} \cong \Coslice{\mathrm{VB}^{\mathbb{T}}_{M, \Cx}}{T}{L} \cong \sqrt[H]{L}$ to $\Coslice{\Prin(M, \Cx^{\times})}{\widehat{\uprho}_2}{Q'} \cong \Coslice{\mathrm{VB}^{\times}_{M, \Cx}}{T}{L} \cong \sqrt{L}$. This functor just takes $S \twoheadrightarrow L$ and forgets about the Hermitian metrics on both of them.

\begin{lemma}\label{hermitiansqrt}
This functor is an equivalence.
\end{lemma}

\begin{proof}
It is faithful by Lemma \ref{coslice}. We prove essential surjectivity. Let $S \twoheadrightarrow L$ be a bundle map that is quadratic on each fiber. Here $S$ is just a line bundle and $L$ is a Hermitian line bundle. We want to endow $S$ with a Hermitian metric so that our quadratic map becomes length-squaring. Observe that $Q' \cong L \mysetminus \set{0}$ as principal $\Cx^{\times}$-bundles (this follows from the discussion on page \pageref{principal_isomorphic}). A Hermitian metric in $L$ is the same as a $\mathbb{T}$-reduction $U \subset L \mysetminus \set{0}$. In each fiber, this is just the $\mathbb{T}$-orbit of unit-length vectors. Its preimage under $S \mysetminus \set{0} \twoheadrightarrow L \mysetminus \set{0}$ is a $\mathbb{T}$-reduction of the former, hence a Hermitian metric in $S$ that we are looking for. For surjectivity on morphisms, just observe that if $S$ and $S'$ are two Hermitian square roots of $L$ and $S \isoto S'$ is an isomorphism of line bundles respecting their quadratic maps, then it must be length-preserving.
\end{proof}

Finally, assume that $M$ is a complex manifold and $L$ is in fact a holomorphic line bundle or, equivalently, $Q'$ is a holomorphic principle $\Cx^{\times}$-bundle. By Lemma \ref{coslice}, we have a functor from $\Coslice{\Prin^{\hol}(M, \Cx^{\times})}{\widehat{\uprho}_2}{Q'} \cong \Coslice{\mathrm{VB}^{\times, \hol}_{M, \Cx}}{T}{L} \cong \sqrt[\hol]{L}$ to $\Coslice{\Prin(M, \Cx^{\times})}{\widehat{\uprho}_2}{Q'} \cong \Coslice{\mathrm{VB}^{\times}_{M, \Cx}}{T}{L} \cong \sqrt{L}$, which simply forgets about holomorphic structures.

\begin{lemma}\label{holomorphicsqrt}
This functor is an equivalence.
\end{lemma}

\begin{proof}
It is injective on morphisms by Lemma \ref{coslice}. For surjectivity, let $P \twoheadrightarrow Q'$ and $P' \twoheadrightarrow Q'$ be two holomorphic reductions along the morphism $\uprho_2 \colon \Cx \to \Cx$. Both of them are quotients by $\mathbb{Z}/2\mathbb{Z}$ and hence holomorphic covering maps. Any isomorphism $P \isoto P'$ between them as between smooth principal $\Cx^{\times}$-bundles will be a covering isomorphism and thus holomorphic. For essential surjectivity, let $P \twoheadrightarrow Q'$ be a smooth reduction. This is a smooth covering map, so the complex structure on $Q'$ can be lifted to $P$, making it into a complex manifold and making the covering map holomorphic. We need to check that the right action of $\Cx^{\times}$ on $P$ is holomorphic. Consider the following diagram:
$$
\xymatrix{
P \times \Cx^\times \ar[r]_-{\sim} \ar@{->>}[d] & P \ar@{->>}[d] \\
Q' \times \Cx^\times \ar[r]_-{\sim} & Q'
}
$$
Here the top arrow is the original action of $\Cx^\times$ on $P$, whereas the bottom one is the squared action, i.e. $\uplambda \in \Cx^\times$ acts on $Q'$ by $\uplambda^2$. This way, the diagram becomes commutative. As the vertical arrows are holomorphic covering maps and the bottom arrow is holomorphic, so must be the top one.
\end{proof}

This completes the proof of Proposition \ref{spinsquare}.
\end{proof}

\begin{sepcorollary}[spinsquare]
Let $M$ be a Riemann surface with a fixed Riemannian metric $g$ within its conformal class. Then the category of spin structures on $M$ is naturally isomorphic to the categories of Hermitian, smooth, and holomorphic square roots of $TM$.
\end{sepcorollary}

\begin{remark}
\begin{enumerate}
    \item We see that a spin structure is, in a sense, a conformally invariant notion in dimension (rank) 2. Given a Riemann surface $M$, some authors talk about spin structures on $M$ even when no Riemannian metric is fixed, meaning (holomorphic/smooth) square roots of $TM$. We will also do this occasionally.
    \item One could also introduce the obvious notion of a Hermitian holomorphic square root of a Hermitian holomorphic line bundle. It follows easily from the above discussion that if $M$ is a complex manifold and $L$ is a Hermitian holomorphic line bundle over $M$, then $\sqrt[H,\hol]{L} \cong \sqrt[\hol]{L}$ simply by forgetting about the Hermitian metric.
\end{enumerate}
\end{remark}

Finally, we discuss the relationship between square roots of a complex line bundle and its dual. Let $M$ be a smooth manifold and $L$ a complex line bundle over $M$. It is easy to see by repeating the arguments from the proof of Proposition \ref{spinsquare} and applying Lemma \ref{coslice} to the complex dual functor $\mathrm{VB}^{\times}_{M, \Cx} \to \mathrm{VB}^{\times}_{M, \Cx}, \, S \mapsto S^{* 1,0}$, (made covariant by sending isomorphisms to not just their duals but the inverses of their duals) that $\sqrt{L} \cong \sqrt{\displaystyle L^{* 1,0}}$. From a practical point of view, to give a square root $S \twoheadrightarrow L$ is the same as to give an isomorphism $S \otimes S \isoto L$, which gives rise to $S^{* 1,0} \otimes S^{* 1,0} \isoto L^{* 1,0}$, i.e. a square root $S^{* 1,0} \twoheadrightarrow L^{* 1,0}$. 

\begin{agreement}\label{notationsqrt}
Given $\upxi \in S_p^{* 1,0}$, let us agree to denote its image in $L^{* 1,0}_p$ by $\boldsymbol{\upxi^2}$. Given another $\upxi' \in S_p^{* 1,0}$, we write $\boldsymbol{\upxi \upxi'}$ for $\frac{1}{2}((\upxi + \upxi')^2 - \upxi^2 - \upxi'^2)$ (the same for $S \twoheadrightarrow L$ and $S^{* 0,1} \twoheadrightarrow L^{* 0,1}$). Also, given $v \in L_p$, we write $\boldsymbol{\sqrt{v}}$ for any of its two (or one, if $v = 0$) preimages in $S_p$ (they differ by a sign). Although this is a bit ambiguous, our considerations will always be independent of this choice. The only rule is that, throughout a single argument/calculation, $\sqrt{v}$ always stands for \textit{the same} preimage of $v$.
\end{agreement}

A simple computations shows that, given $\upxi, \upxi',$ and $v$ as above, the following holds: 
\begin{equation}\label{root}
\upxi^2 (v) = \upxi(\sqrt{v})^2, \quad \upxi \upxi' (v) = \upxi(\sqrt{v}) \upxi'(\sqrt{v}).
\end{equation}
If we apply the above discussion to the complex antidual functor $S \mapsto S^{* 0,1}$, we get an equivalence $\sqrt{L} \cong \sqrt{\displaystyle L^{* 0,1}}$, which admits a similar explicit description: a square root $S \twoheadrightarrow L$ yields $S^{* 0,1} \otimes S^{* 0,1} \twoheadrightarrow L^{* 0,1}$, which is the same as a square root $S^{* 0,1} \twoheadrightarrow L^{* 0,1}$. The discussion for the complex dual remains true in the holomorphic setting: $\sqrt[\hol]{L} \cong \sqrt[\hol]{\displaystyle L^{* 1,0}}$. If $L$ is Hermitian, we endow $L^{* 1,0}$ and $L^{* 0,1}$ with the corresponding Hermitian metrics, so  $\sqrt[H]{L} \cong \sqrt[H]{\displaystyle L^{* 1,0}} \cong \sqrt[H]{\displaystyle L^{* 0,1}}$. 

\begin{example}
If $M$ is a Riemann surface, then there is an equivalence $\sqrt[\hol]{\displaystyle T^{* 1,0}M} \cong \sqrt[\hol]{TM}$.
\end{example}

\begin{observation}
Assume that $M$ is a compact Riemann surface. By the divisor -- line bundle correspondence, a holomorphic square root of $T^*M$ is the same as a divisor class $[D]$ such that $2[D] = K$. Such a divisor class is called a \textbf{theta-characteristic}. It follows easily from the discussion above that the set of equivalence classes of spin structures on $M$ is in one-to-one correspondence with the set of theta-characteristics on $M$.
\end{observation}

\subsection{Spinor representations}

It turns out that if $M$ is a Riemann surface, any Weierstrass representation of $M$ (in particular, any conformal immersion of $M$ into $\Rl^3$) induces a spin structure on it.

\begin{fact}
Let $M$ be a smooth manifold and $Q$ a principal $\SO(n)$-bundle over $M$. Then $Q$ admits a spin structure if and only if $w_2(Q) = 0$. If this is the case, the equivalence classes of spin structures on $Q$ are in bijective correspondence with $H^1(M, \mathbb{Z}/2\mathbb{Z})$.
\end{fact}

The proof for $n \geqslant 3$ can be found in \cite{friedrichbook}, but since the case of interest for us is $n = 2$, so we refer to \cite{borel-hirzebruch-II} and \cite{milnorspin}.

Now, let $M$ be a connected compact Riemann surface of genus $g$. The second Stiefel-Whitney class of $M$ is the $\mod 2$ reduction of its Euler class. The integral of the latter (the Euler number) is equal to the Euler characteristic of $M$, which is $2-2g$ and hence even. Therefore, the $\mod 2$ reduction of the Euler class vanishes. It follows that any compact Riemann surface $M$ admits spin structures (or, to be more precise, $TM$ has square roots), and the number of their equivalence classes equals $\# H^1(M, \mathbb{Z}/2\mathbb{Z}) = 2^{2g}$. In particular, $\Cx P^1$ admits a unique spin structure up to equivalence. In other words, $T^*\Cx P^1$ has a unique square root up to equivalence, and, as we already know from Section \ref{technical}, it is nothing but the tautological line bundle over $\Cx P^1$.

Now, let $M$ be any Riemann surface, and let $(\upomega, \upeta)$ be a Weierstrass representation of $M$. Let us look at it as a bundle map from $TM \to M$ to $T^*\Cx P^1 \to \Cx P^1$. Then $\upomega$ gives an isomorphism of complex line bundles $TM \isoto \upeta^* (T^*\Cx P^1)$ over $M$. We can pull the square root $\mathcal{O}_{\Cx P^1}(-1)$ of $T^*\Cx P^1$ back by $\upeta$ as well. This gives a square root of $\upeta^* (T^*\Cx P^1)$ and hence of $TM$. Altogether, we have the following picture:
\begin{equation}\label{two_squares}
\xymatrix{
S \ar[r]^-{\sqrt{\upomega}} \ar@{->>}[d] & \mathcal{O}_{\Cx P^1}(-1) \ar@{->>}[d] \\
TM \ar[r]^-{\upomega} \ar[d] & T^*\Cx P^1 \ar[d] \\
M \ar[r]^-{\upeta} & \Cx P^1,}
\end{equation}
where $S$ stands for the induced square root $\upeta^*(\mathcal{O}_{\Cx P^1}(-1))$ of $TM$. Clearly, $(\sqrt{\upomega}, \upeta)$ is a nonvanishing bundle map, so, by the observation on page \pageref{obsbundle}, it gives a pair of smooth sections $u, v$ of $S^{* 1,0}$ determined by the formula $\sqrt{\upomega} = (\upeta \circ \uppi_S, (u, v))$. These sections never vanish simultaneously.

\begin{agreement}\label{sqrt}
We endow $S$ with the unique holomorphic structure making it a holomorphic square root of $TM$. Moreover, since a Weierstrass representation is fixed, we have a Hermitian metric $h$ in $TM$. Let $\boldsymbol{\sqrt{h}}$ stand for the lifted Hermitian metric on $S$. This way, $S$ becomes a Hermitian holomorphic square root. In particular, this gives complex-linear isomorphisms $S^{* 1,0} \cong \overline{S}$ and $S \cong S^{* 0,1}$. Implicitly assuming this isomorphism, we will refer to $S^{* 1,0}$ and $S^{* 0,1}$ as the bundles of positive and negative spinors, respectively. In this regard, $u$ and $v$ are positive spinors. Observe that the complex conjugation inside $S^*_\Cx$ interchanges the subbundles of positive and negative spinors. Also note that, in general, $\sqrt{\upomega}$ is only smooth (because so is $\upomega$). Therefore, $u$ and $v$ may not be holomorphic sections of $S^{* 1,0}$.
\end{agreement}

\begin{definition}
Let $M$ be a Riemann surface, and let $S \twoheadrightarrow TM$ be a square root. By the observation on page \pageref{obsbundle}, there is a one-to-one correspondence:
$$
\left\{
\begin{gathered}
\text{nondegenerate bundle maps $(\upbeta, \upeta)$} \\
\text{from} \; S \to M \; \text{to} \; \mathcal{O}_{\Cx P^1}(-1) \to \Cx P^1
\end{gathered}
\right\} \rightleftarrows \left\{
\begin{gathered}
\text{pairs $(u, v)$ of sections of $S^{* 1,0}$} \\
\text{that do not vanish simultaneously}
\end{gathered}
\right\}
$$
A choice of a square root $S$ of $TM$ together with an element of any of these two sets is called a \textbf{spinor representation of} $\boldsymbol{M}$. Since the definition involves a choice of a square root, a suitable notion of equivalence is in order. We say that two spinor representations $(S_1, \upbeta_1, \upeta_1)$ and $(S_2, \upbeta_2, \upeta_2)$ are \textbf{equivalent} if there is an equivalence $\upalpha \colon S_1 \isoto S_2$ of square roots of $TM$ such that $\upbeta_2 \circ \upalpha = \upbeta_1$ or, in terms of pairs of sections $(u_1, v_1)$ and $(u_2, v_2)$, $\upalpha^* u_1 = u_2, \upalpha^* v_1 = v_2$. In particular, we must have $\upeta_1 = \upeta_2$. Given a \textbf{spinor representation} $(S, \upbeta, \upeta)$, it is called \textbf{holomorphic} if $\upbeta$ is holomorphic (with respect to the unique holomorphic structure on $S$ making it into a holomorphic square root). This is the same as requiring $u$ and $v$ to be holomorphic sections of $S^{* 1,0}$. In particular, $\upeta$ must be holomorphic.
\end{definition}

\begin{example}
If $(S, u, v)$ is a spinor representation of $M$, then it is equivalent to $(S, -u, -v)$.
\end{example}

We have already seen that a Weierstrass representation induces a spinor representation (with the same map $\upeta \colon M \to \Cx P^1$). Conversely, given, a spinor representation $(S, \upbeta, \upeta)$, the map $\upbeta \colon S \to \mathcal{O}_{\Cx P^1}(-1)$ clearly passes to a nonvanishing bundle map $\upomega \colon TM \to T^*\Cx P^1$ covering $\upeta$, and equivalent spinor representations give the same $\upomega$.

\begin{conclusion}
Sending $(\upomega, \upeta)$ to $(S = \upeta^*(\mathcal{O}_{\Cx P^1}(-1)), \sqrt{\upomega}, \upeta)$ gives a one-to-one correspondence between the set of Weierstrass representations and the set of equivalence classes of spinor representations of $M$. Under this correspondence, holomorphic Weierstrass representations correspond precisely to equivalence classes of holomorphic spinor representations.
\end{conclusion}

Given a Weierstrass representation $(\upomega, \upeta)$ of $M$, we can express $\upomega$, thought of as a $\Cx^3$-valued (1,0)-form, in terms of the positive spinors $u$ and $v$ of the corresponding spinor representation. Recall that, according to the agreement on page \pageref{notationsqrt}, we have smooth (1,0)-forms $u^2, v^2$, and $uv$.

\begin{lemma}\label{wrviasr}
$\upomega = (u^2 - v^2, i(u^2 + v^2), 2uv).$
\end{lemma}

\begin{proof}
This is basically the formula for $\uptheta$ from Section \ref{technical} rewritten in a fancy way. Recall that when we want to think of $\upomega$ as of a triple of (1,0)-forms on $M$, we think of the bundle map $(\upomega, \upeta)$ as of a map to $\mathcal{O}_{[Q]}(-1) \to [Q]$. Let $w \in T_pM$. Then
\begin{multline*}
    (u^2 - v^2, i(u^2 + v^2), 2uv)(w) = \left( \begin{gathered} u(\sqrt{w})^2 - v(\sqrt{w})^2 \\ 
    i(u(\sqrt{w})^2 + v(\sqrt{w})^2) \\
    2u(\sqrt{w})v(\sqrt{w}) \end{gathered} \right) \\
    = \left( \begin{gathered} \uppi_1(\sqrt{\upomega}(\sqrt{w}))^2 - \uppi_2(\sqrt{\upomega}(\sqrt{w}))^2 \\
    i(\uppi_1(\sqrt{\upomega}(\sqrt{w}))^2 + \uppi_2(\sqrt{\upomega}(\sqrt{w}))^2) \\
    2\uppi_1(\sqrt{\upomega}(\sqrt{w}))\uppi_2(\sqrt{\upomega}(\sqrt{w})) \end{gathered} \right) = \uptheta \circ \uppi_{\Cx^2} \circ \sqrt{\upomega}(\sqrt{w}) = \uppi_{\Cx^3} \circ \upgamma \circ \sqrt{\upomega}(\sqrt{w}),
\end{multline*}
where $\uppi_1$ and $\uppi_2$ are the projections of $\Cx P^1 \times \Cx^2$ onto the coordinate axis of $\Cx^2$, $\uppi_{\Cx^2}$ is the projection of it onto the whole $\Cx^2$, $\uppi_{\Cx^3}$ is the projection of $[Q] \times \Cx^3$ onto $\Cx^3$, and $\upgamma \colon \mathcal{O}_{\Cx P^1}(-1) \twoheadrightarrow \mathcal{O}_{[Q]}(-1) \subset [Q] \times \Cx^3$ is a map defined in Section \ref{technical}. But the last expression in the computation above is precisely $\upomega(w)$ by the commutativity of the upper square on the diagram \eqref{two_squares} and the fact that $\upgamma$ coincides with the map $\mathcal{O}_{\Cx P^1}(-1) \twoheadrightarrow T^*\Cx P^1$ in \eqref{two_squares} under the isomorphism $T^*\Cx P^1 \cong T^*[Q] \cong \mathcal{O}_{[Q]}(-1)$.
\end{proof}

\subsection{Recovering an immersion from its spinor representation}

Here we show how to express the integrability and period conditions on a Weierstrass representation of a Riemann surface in terms of the corresponding spinor representation. After that, we show how a spinor representation gives rise to an entire family of spinor representations all of whose elements satisfy the integrability and/or period condition -- as soon as the initial spinor representation does. We then describe the corresponding family of Weierstrass representations and conformal immersions.

Let $M$ be a Riemann surface, $(S, u, v)$ a spinor representation of $M$, and $\upomega$ the corresponding Weierstrass representation. First of all, the Hermitian metric $\sqrt{h}$ of $S$ can be described directly in terms of $u$ and $v$. Indeed, given $w \in S_p$, we compute using Lemma \ref{wrviasr}:
\begin{align*}
\sqrt{h}(w) = \sqrt{h(w^2)} &= ||\upomega(w^2)|| \\ 
&= \sqrt{|(u^2 - v^2)(w^2)|^2 + |(u^2 + v^2)(w^2)|^2 + |2uv(w^2)|^2} \\
&= \sqrt{|u(w)^2 - v(w)^2|^2 + |u(w)^2 + v(w)^2|^2 + 4|u(w)v(w)|^2} \\
&= \sqrt{2|u(w)|^4 + 2|v(w)|^4 + 4|u(w)|^2|v(w)|^2} \\
&= \sqrt{2}(|u(w)|^2 + |v(w)|^2) \\
&= \sqrt{2}||\sqrt{\upomega}(w)||^2.
\end{align*}
Next, it follows from Lemma \ref{wrviasr} that
$$
\upmu_+ = -2v^2, \quad \upmu_- = 2u^2.
$$
In particular, $M_+$ is the zero locus of $v$ and $M_-$ is the zero locus of $u$. Moreover,
$$
\upnu_+ = -\frac{u}{v} \; (\text{over $U_+$}), \quad \upnu_- = \frac{v}{u} \; (\text{over $U_-$}),
$$
or, more formally, $u = -\upnu_+ v, v = \upnu_- u$.

To proceed further, we need to establish some line bundle isomorphisms and discuss the Dirac operator. Since $\upomega$ endows $M$ with a Riemannian metric and a spin structure, there is a Dirac operator acting on the sections of the spinor bundle, i.e., in our case, of $\overline{S} \oplus S \cong S^{* 1,0} \oplus S^{* 0,1}$. Since it sends positive spinors to negative ones and vice versa, it interchanges the (sections of these) summands. Since the case of Riemann surfaces is so special, it should come as no surprise that the Dirac operators allows an alternative, simpler expression.

Observe that $\sqrt{h}$ is a nonvanishing section of $S^{* 1,0} \otimes_{\Cx} S^{* 0,1}$. Multiplication by $\sqrt{h}$ gives an isomorphism of complex line bundles $S^{* 1,0} \isoto S^{* 1,0} \otimes S^{* 1,0} \otimes S^{* 0,1}, \upxi \mapsto \upxi \otimes \sqrt{h}$. The latter bundle is isomorphic to $T^{* 1,0}M \otimes S^{* 0,1}$. We call the resulting isomorphism $\uptau_{1,0} \colon S^{* 1,0} \isoto T^{* 1,0}M \otimes S^{* 0,1}$. Let $z$ be a local holomorphic coordinate on $M$ and write $\upomega = \widetilde{\upomega} dz$. Then $\frac{\partial}{\partial x}$ locally trivializes $TM$, and it can be lifted to two local sections of $S$. We pick one of them and call it $\sqrt{\frac{\partial}{\partial x}}$. The local section of $T^{* 1,0}M$ dual to $\frac{\partial}{\partial x}$ is $dz$. It can be lifted to two local sections of $S^{* 1,0}$, but only one of them is dual to $\sqrt{\frac{\partial}{\partial x}}$, and we call it $\sqrt{dz}$ (the other one is $-\sqrt{dz}$). Observe that the local section $\sqrt{d\overline{z}} = \overline{\sqrt{dz}}$ of $S^{* 0,1}$ is a lift of the local section $d\overline{z}$ of $T^{* 0,1}M$. We already know that $h = ||\widetilde{\upomega}||^2 dz \otimes d\overline{z}$ (see the proof of Proposition \ref{mean}), so $\sqrt{h} = ||\widetilde{\upomega}|| \sqrt{dz} \otimes \sqrt{d\overline{z}}$. The isomorphism $\uptau_{1,0}$ locally looks like $f \sqrt{dz} \mapsto f ||\widetilde{\upomega}|| dz \otimes \sqrt{d\overline{z}}$. Similarly, we have $\uptau_{0,1} \colon S^{* 0,1} \isoto S^{* 0,1} \otimes S^{* 1,0} \otimes S^{* 0,1} \cong T^{* 0,1}M \otimes S^{* 1,0}$, where the first map is the multiplication by $\sqrt{h}$. In coordinates, $\uptau_{0,1} \colon f \sqrt{d\overline{z}} \mapsto f ||\widetilde{\upomega}|| d\overline{z} \otimes \sqrt{dz}$.

Since $S^{* 1,0}$ is a holomorphic line bundle, we have the Dolbeault operator $\overline{\partial} \colon \Upomega^0(M, S^{* 1,0}) \to \Upomega^{0, 1}(M, S^{* 1,0})$. The latter is the space of sections of the bundle $T^{* 0,1}M \otimes_{\Cx} S^{* 1,0}$, which is isomorphic to $S^{* 0,1}$ via $\uptau_{0,1}^{-1}$. The resulting operator $\Upgamma(S^{* 1,0}) \to \Upgamma(S^{* 0,1})$ is exactly the Dirac operator, and in coordinates it looks like $f \sqrt{dz} \mapsto \frac{f_{\overline{z}}}{||\widetilde{\upomega}||} \sqrt{d\overline{z}}$. Similarly, $S^{* 0,1}$ is conjugate to $S^{* 1,0}$, so it is an antiholomorphic line bundle and hence there is an analogue of the Dolbeault operator $\partial \colon \Upomega^0(M, S^{* 0,1}) \to \Upomega^{1, 0}(M, S^{* 0,1})$. Since $T^{* 1,0}M \otimes_{\Cx} S^{* 0,1}$ is isomorphic to $S^{* 1,0}$ via $\uptau_{1,0}^{-1}$, we have the other half of the Dirac operator: $\Upgamma(S^{* 0,1}) \to \Upgamma(S^{* 1,0}), f \sqrt{d\overline{z}} \mapsto \frac{f_{z}}{||\widetilde{\upomega}||} \sqrt{dz}$. We use the standard letter $D$ for both of them. Note that $D$ commutes with the complex conjugation. Finally, observe that the same multiplication by $\sqrt{h}$ gives yet another isomorphism $S^{* 1,0} \otimes S^{* 0,1} \isoto S^{* 1,0} \otimes S^{* 0,1} \otimes S^{* 1,0} \otimes S^{* 0,1} \cong T^{* 1,0}M \otimes T^{* 0,1}M$. In coordinates, $f \sqrt{dz} \otimes \sqrt{d\overline{z}} \mapsto f ||\widetilde{\upomega}|| dz \otimes d\overline{z}$.

We proceed to express the integrability and period conditions on $\upomega$ in terms of its spinor representation. We start with the latter because it is easier.

\begin{proposition}\label{spinorperiods}
Let $\upomega$ be a Weierstrass representation of a Riemann surface $M$, and let $(S, u, v)$ be the corresponding spinor representation. Then:
$$
\begin{gathered} \text{the periods of} \; \upomega \; \text{are} \\
\text{purely imaginary} \end{gathered} \; \Longleftrightarrow \; \left\{ \begin{aligned} &u^2 \; \text{and} \; v^2 \; \text{have conjugate periods,} \\
&uv \; \text{has purely imaginary periods}\end{aligned} \right.
$$
\end{proposition}

\begin{proof}
Follows elementary from Lemma \ref{wrviasr}.
\end{proof}

Now to the harder part. Let us take a closer look at the line bundle $S^{* 1,0} \otimes S^{* 0,1}$. It is the bundle of $\Cx$-$\frac{3}{2}$-linear forms on $S$. Consequently, it has a real structure given by $s \otimes t \mapsto \overline{t \otimes s} = \overline{t} \otimes \overline{s}$. Pointwise, the $+1$- and $-1$-eigenspaces of this involution are the subspaces of Hermitian and skew-Hermitian $\Cx$-$\frac{3}{2}$-linear forms, respectively, and the multiplication by $i$ interchanges them. Hermitian forms are exactly those whose quadratic forms are real, while the quadratic forms of skew-Hermitian forms are imaginary. When bunched together, these eigenspaces form two real line subbundles of $S^{* 1,0} \otimes S^{* 0,1}$. We will denote the projections onto them within $S^{* 1,0} \otimes S^{* 0,1}$ by $\Re$ and $\Im$: 
$$
\Re(s \otimes t) = \frac{1}{2}(s \otimes t + \overline{t} \otimes \overline{s}), \hspace{1em} \Im(s \otimes t) = \frac{1}{2}(s \otimes t - \overline{t} \otimes \overline{s}).
$$
It is important that these differ from the operations of taking the real and imaginary parts inside $S^*_\Cx \otimes S^*_\Cx$.

\begin{proposition}\label{spinorintegrability}
Let $\upomega$ be a Weierstrass representation of a Riemann surface $M$, and let $(S, u, v)$ be the corresponding spinor representation. Then the integrability condition on $\upomega$ is equivalent to the following:
$$
\left\{ \begin{aligned} &v \otimes Dv = - D\overline{u} \otimes \overline{u} \\
&u \otimes Dv + v \otimes Du = D\overline{v} \otimes \overline{u} + D\overline{u} \otimes \overline{v} \end{aligned} \right.
$$
\end{proposition}

\begin{proof}
First of all, all the action in the equations above is taking place in the space of sections of $S^{* 1,0} \otimes S^{* 0,1}$. Let $z$ be a local holomorphic coordinate. As usual, we write $\upomega = \widetilde{\upomega} dz$. Let $u = \widetilde{u} \sqrt{dz}, \, v = \widetilde{v} \sqrt{dz}$. Then
$$
Du = \frac{\widetilde{u}_{\overline{z}}}{||\widetilde{\upomega}||} \sqrt{\overline{dz}}, \quad Dv = \frac{\widetilde{v}_{\overline{z}}}{||\widetilde{\upomega}||} \sqrt{\overline{dz}}.
$$
We have:
$$
d\upomega = \overline{\partial} \upomega = \overline{\partial} \left( \begin{gathered} (\widetilde{u}^2 - \widetilde{v}^2) dz \\
i(\widetilde{u}^2 + \widetilde{v}^2) dz \\
2\widetilde{u}\widetilde{v} dz \end{gathered} \right) = -2 \left( \begin{gathered} \widetilde{u}\widetilde{u}_{\overline{z}} - \widetilde{v}\widetilde{v}_{\overline{z}} \\
i(\widetilde{u}\widetilde{u}_{\overline{z}} + \widetilde{v}\widetilde{v}_{\overline{z}}) \\ \widetilde{u}_{\overline{z}}\widetilde{v} + \widetilde{u}\widetilde{v}_{\overline{z}}\end{gathered} \right) dz \wedge d\overline{z}.
$$
This is a triple of sections of the complex line bundle $\extp^{1,1}T^*M = \extp^2 T^*_\Cx M$. This bundle is isomorphic to $T^{* 1,0}M \otimes T^{* 0,1}M$, where the isomorphism simply replaces $\wedge$ by $\otimes$. From this point of view, as a triple of sections of $T^{* 1,0}M \otimes T^{* 0,1}M$, $d\upomega$ is given by
\begin{equation}\label{triple1}
-2 \left( \begin{gathered} \widetilde{u}\widetilde{u}_{\overline{z}} - \widetilde{v}\widetilde{v}_{\overline{z}} \\
i(\widetilde{u}\widetilde{u}_{\overline{z}} + \widetilde{v}\widetilde{v}_{\overline{z}}) \\ \widetilde{u}_{\overline{z}}\widetilde{v} + \widetilde{u}\widetilde{v}_{\overline{z}}\end{gathered} \right) dz \otimes d\overline{z}.
\end{equation}
But under the isomorphism $T^{* 1,0}M \otimes T^{* 0,1}M \cong S^{* 1,0} \otimes S^{* 0,1} \otimes S^{* 1,0} \otimes S^{* 0,1}$, this corresponds to
\begin{multline}\label{triple2}
    -2 \left( \begin{gathered} \widetilde{u}\widetilde{u}_{\overline{z}} - \widetilde{v}\widetilde{v}_{\overline{z}} \\
    i(\widetilde{u}\widetilde{u}_{\overline{z}} + \widetilde{v}\widetilde{v}_{\overline{z}}) \\
    \widetilde{u}_{\overline{z}}\widetilde{v} + \widetilde{u}\widetilde{v}_{\overline{z}} \end{gathered} \right) \sqrt{dz} \otimes \sqrt{d\overline{z}} \otimes \sqrt{dz} \otimes \sqrt{d\overline{z}} \\
    = -2 \left( \begin{gathered} \widetilde{u}\frac{\widetilde{u}_{\overline{z}}}{||\widetilde{\upomega}||} - \widetilde{v}\frac{\widetilde{v}_{\overline{z}}}{||\widetilde{\upomega}||} \\
    i(\widetilde{u}\frac{\widetilde{u}_{\overline{z}}}{||\widetilde{\upomega}||} + \widetilde{v}\frac{\widetilde{v}_{\overline{z}}}{||\widetilde{\upomega}||}) \\
    \frac{\widetilde{u}_{\overline{z}}}{||\widetilde{\upomega}||}\widetilde{v} + \widetilde{u}\frac{\widetilde{v}_{\overline{z}}}{||\widetilde{\upomega}||} \end{gathered} \right) \sqrt{dz} \otimes \sqrt{d\overline{z}} \otimes \left( ||\widetilde{\upomega}|| \sqrt{dz} \otimes \sqrt{d\overline{z}} \right) \\
    = -2 \left( \begin{gathered} u \otimes Du - v \otimes Dv \\
    i(u \otimes Du + v \otimes Dv) \\
    u \otimes Dv + v \otimes Du \end{gathered} \right) \otimes \sqrt{h}. \quad
\end{multline}
Note that, although the bundles $T^{* 1,0}M \otimes T^{* 0,1}M$ and $\extp^2 T^*_\Cx M$ are isomorphic, their sections act on vector fields differently. In this regard, the former is the bundle of $\Cx$-$\frac{3}{2}$-linear forms, whereas the latter is the bundle of complex-valued 2-forms. Pointwise, the isomorphism sends a $\frac{3}{2}$-linear form $\upalpha$ to $\upalpha - \upalpha^t$. Clearly, it sends Hermitian $\frac{3}{2}$-forms to imaginary 2-forms and skew-Hermitian $\frac{3}{2}$-forms to real 2-forms. The integrability condition requires $d\upomega$ to be imaginary, which is equivalent to asking the triple \eqref{triple1} of fields of $\Cx$-$\frac{3}{2}$-linear forms on $TM$ to be Hermitian. As we observed above, this means that the corresponding field of quadratic forms should be real-valued. But according to \eqref{root}, this means that, for any $w \in S$, \eqref{triple2} should output a real number when all of its four arguments are taken to be $w$. Since $\sqrt{h}$ is Hermitian, $\sqrt{h}(w,w)$ is real. Thus, the integrability condition is equivalent to asking
$$
\left( \begin{gathered} u \otimes Du - v \otimes Dv \\
i(u \otimes Du + v \otimes Dv) \\
u \otimes Dv + v \otimes Du \end{gathered} \right)
$$
to be a triple of fields of \textit{Hermitian} $\Cx$-$\frac{3}{2}$-linear forms on $S$. In other words, $\upomega$ satisfies the integrability condition if and only if:
\begin{multline*}
\begin{cases} \Im(u \otimes Du - v \otimes Dv) = 0 \\
\Re(u \otimes Du + v \otimes Dv) = 0 \\
\Im(u \otimes Dv + v \otimes Du) = 0 \end{cases} \; \Leftrightarrow \; \begin{cases} \Im(u \otimes Du) = \Im(v \otimes Dv) = 0 \\
\Re(u \otimes Du) = -\Re(v \otimes Dv) = 0 \\
\Im(u \otimes Dv + v \otimes Du) = 0 \end{cases} \; \Leftrightarrow \\
\Leftrightarrow \; \begin{cases} v \otimes Dv = - D\overline{u} \otimes \overline{u} \\
u \otimes Dv + v \otimes Du = D\overline{v} \otimes \overline{u} + D\overline{u} \otimes \overline{v} \end{cases}
\end{multline*}
which was to be proved.
\end{proof}

Finally, we describe how a spinor representation naturally produces a family of those, and then we look at this family through the lens of Weierstrass representations and conformal immersions. Given a Riemann surface $M$, there is an action of $\GL(2, \Cx)$ on the set of equivalence classes of spinor representations of $M$ (and hence on the set of Weierstrass representations):
$$
\GL(2, \Cx) \ni \begin{pmatrix} a & b \\ c & d \end{pmatrix} \colon (S, (u, v)) \mapsto (S, \begin{pmatrix} a & b \\ c & d \end{pmatrix} \begin{pmatrix} u \\ v \end{pmatrix}).
$$
If $T = \begin{pmatrix} a & b \\ c & d \end{pmatrix}$, let us write $T \begin{pmatrix} u \\ v \end{pmatrix}$ by $(u_T, v_T)$. In terms of Weierstrass representations, let us denote the image of $\upomega$ under $T$ by $\upomega_T$. This action obviously sends holomorphic spinor/Weierstrass representations to holomorphic ones. Also note that this action is not effective, as its kernel is $\set{\pm 1}$.

The subgroup $\mathbb{H}^{\times} \cong \Rl^+ \times \SU(2) = \set{ \begin{pmatrix} a & -\overline{b} \\ b & \overline{a} \end{pmatrix} \mid a, b \in \Cx \; \text{not both zero} } \subset \GL(2, \Cx)$ is of particular salience due to the following

\begin{proposition}\label{actionquaternions}
The action of $\Rl^+ \times SU(2)$ on Weierstrass representations preserves both the integrability and period conditions. That is, if a Weierstrass representation $\upomega$ of $M$ satisfies the integrability and/or period condition, then so does $\upomega_T$ for every $T \in \Rl^+ \times \SU(2)$.
\end{proposition}

\begin{proof}
Let $\upomega$ be a Weierstrass representation of $M$, let $(S, u, v)$ be the corresponding spinor representation, and let $T = \begin{pmatrix} a & -\overline{b} \\ b & \overline{a} \end{pmatrix}$. We compute:
$$
u_T = au - \overline{b}v, \quad v_T = bu + \overline{a}v,
$$
so
$$
\begin{cases} 
u_T^2 = a^2 u^2 - 2a\overline{b}uv + \overline{b}^2 v^2 \\ v_T^2 = b^2 u^2 + 2\overline{a}b uv + \overline{a}^2 v^2 \\
u_T v_T = abu^2 + (|a|^2 - |b|^2)uv - \overline{ab}v^2
\end{cases}
$$
Assume that $\upomega$ satisfies the period condition, and let $\upgamma$ be any piecewise smooth loop in $M$. By Proposition \ref{spinorperiods}, we have:
$$
\overline{\int_{\upgamma} u^2} = \int_{\upgamma} v^2, \quad \int_{\upgamma} uv \in i\Rl.
$$
But then
\begin{align*}
\overline{\int_{\upgamma} u_T^2} &= \overline{a^2 \int_{\upgamma} u^2} - \overline{2a\overline{b} \int_{\upgamma} uv} + \overline{\overline{b}^2 \int_{\upgamma} v^2} \\
&= \overline{a}^2 \int_{\upgamma} v^2 + 2\overline{a}b \int_{\upgamma} uv + b^2 \int_{\upgamma} u^2 \\
&= \int_{\upgamma} v_T^2, \\
\int_{\upgamma} u_T v_T &= ab \int_{\upgamma} u^2 - \overline{ab} \int_{\upgamma} v^2 + (|a|^2 - |b|^2) \int_{\upgamma} uv \\
& = ab \int_{\upgamma} u^2 - \overline{ab \int_{\upgamma} u^2} + (|a|^2 - |b|^2) \int_{\upgamma} uv \\
&= 2i \Im \left[ ab \int_{\upgamma} u^2 \right] + (|a|^2 - |b|^2) \int_{\upgamma} uv \in i\Rl,
\end{align*}
so $\upomega_T$ satisfies the period condition as well. Now assume that $\upomega$ satisfies the integrability condition (regardless of the period condition). By Proposition \ref{spinorintegrability}, this means that
$$
\left\{ \begin{aligned} &v \otimes Dv = - D\overline{u} \otimes \overline{u} \\
&u \otimes Dv + v \otimes Du = D\overline{v} \otimes \overline{u} + D\overline{u} \otimes \overline{v} \end{aligned} \right.
$$
We compute:
\begin{align*}
    v_T \otimes Dv_T &= (bu + \overline{a}v) \otimes (b Du + \overline{a} Dv) \\
    &= b^2 u \otimes Du + \overline{a}b (u \otimes Dv + v \otimes Du) + \overline{a}^2 v \otimes Dv \\
    &= -b^2 D\overline{v} \otimes \overline{v} + \overline{a}b (D\overline{v} \otimes \overline{u} + D\overline{u} \otimes \overline{v}) - \overline{a}^2 D\overline{u} \otimes \overline{u} \\
    &= - (\overline{a}D\overline{u} - b D \overline{v}) \otimes (\overline{a}\overline{u} - b \overline{v}) \\
    &= - D\overline{u_T} \otimes \overline{u_T}, \\
    u_T \otimes Dv_T + v_T \otimes Du_T &= (au - \overline{b}v) \otimes (bDu + \overline{a}Dv) + (bu + \overline{a}v) \otimes (aDu - \overline{b}Dv) \\
    &= 2ab u \otimes Du + (|a|^2 - |b|^2) (u \otimes Dv + v \otimes Du) - 2\overline{ab} v \otimes Dv \\
    &= -2ab  D\overline{v} \otimes \overline{v} + (|a|^2 - |b|^2) (D\overline{v} \otimes \overline{u} + D\overline{u} \otimes \overline{v}) + 2\overline{ab} D\overline{u} \otimes \overline{u} \\
    &= (\overline{b}D\overline{u} + aD\overline{v}) \otimes (\overline{au} - b\overline{v}) + (\overline{a}D\overline{u} - b D \overline{v}) \otimes (\overline{bu} + a\overline{v}) \\
    &= D \overline{v_T} \otimes \overline{u_T} + D \overline{u_T} \otimes \overline{v_T},
\end{align*}
which means that $\upomega_T$ satisfies the integrability condition as well.
\end{proof}

Let us find out how $\upomega_T$ can be expressed via $\upomega$:
\begin{align*}
\upomega_T &= \left( \begin{gathered} (au - \overline{b}v)^2 - (bu + \overline{a}v)^2 \\
i((au - \overline{b}v)^2 + (bu + \overline{a}v)^2) \\
2(au - \overline{b}v) (bu + \overline{a}v) \end{gathered} \right) \\[0.3em]
&= \left( \begin{gathered} (a^2 - b^2)u^2 + 4 \Re (\overline{a} b) uv - \overline{(a^2 - b^2)} v^2 \\
i((a^2 + b^2)u^2 + 4i\Im (\overline{a} b) uv + \overline{(a^2 + b^2)} v^2) \\
2abu^2 + 2(|a|^2 - |b|^2)uv - 2\overline{ab}v^2 \end{gathered} \right) \\[0.4em]
&= \left( \begin{gathered} \frac{1}{2}(a^2 - b^2)(\upomega_1 - i\upomega_2) + 2 \Re (\overline{a} b) \upomega_3 + \frac{1}{2}\overline{(a^2 - b^2)}(\upomega_1 + i\upomega_2) \\
i(\frac{1}{2}(a^2 + b^2)(\upomega_1 -i\upomega_2) + 2i\Im (\overline{a} b)\upomega_3 - \frac{1}{2}\overline{(a^2 + b^2)}(\upomega_1 + i\upomega_2)) \\
ab(\upomega_1 - i\upomega_2) + (|a|^2 - |b|^2)\upomega_3 + \overline{ab}(\upomega_1 + i\upomega_2) \end{gathered} \right) \\[0.35em]
&= \begin{pmatrix} \Re(a^2 - b^2) & \Im(a^2 - b^2) & 2\Re(\overline{a}b) \\[0.3em]
-\Im(a^2 + b^2) & \Re(a^2 + b^2) & -2\Im(\overline{a}b) \\[0.3em]
2\Re(ab) & 2\Im(ab) & |a|^2 - |b|^2 \end{pmatrix} \begin{pmatrix} \upomega_1 \\[0.3em] \upomega_2 \\[0.3em] \upomega_3 \end{pmatrix} = \widehat{T} \upomega.
\end{align*}

It can be easily seen that $\widehat{T}$ lies in the subgroup $\Rl^+ \times \SO(3)$ of conformal orientation-preserving linear automorphisms of $\Rl^3$: its columns are all of length $|a|^2 + |b|^2$ and pairwise orthogonal. Consider the morphism of Lie groups $\Rl^+ \times \SU(2) \twoheadrightarrow \Rl^+ \times \SO(3, \Rl)$ sending $T$ to $\widehat{T}$. Note that this morphism actually splits as the product of $\Rl \xrightarrow{t \mapsto t^2} \Rl$ and a two-sheeted covering map $\SU(2) \twoheadrightarrow \SO(3)$. The group $\Rl^+ \times \SO(3)$ acts on Weierstrass representations in an obvious way. We have shown that the map sending an equivalence class of spinor representation to its corresponding Weierstrass representation is equivariant with respect to the actions of $\Rl^+ \times \SU(2)$ and $\Rl^+ \times \SO(3, \Rl)$ (when these two groups are related by the morphism described above).

Proposition \ref{actionquaternions} together with the computation above imply that the action of $\mathbb{H}^{\times}$ on the set of Weierstrass representations of $M$ induces an action on the set of conformal immersions $M \to \Rl^3$ modulo translations of $\Rl^3$. Under this action, $T \in \mathbb{H}^{\times}$ sends an immersion $\upvarphi$ to $\widehat{T} \circ \upvarphi$. As we have already observed above, $\widehat{T}$ can be decomposed as the product of a (positive) dilation with center at the origin and an element of $\SO(3)$ (hence a rotation around some axis). This action sends minimal immersions to minimal ones because conformal linear automorphisms of $\Rl^3$ are harmonic. This latter fact can also be shown by noting that the action of $\mathbb{H}^{\times}$ on Weierstrass representations preserves the subset of holomorphic Weierstrass representations.

\nocite{bures, osserman, friedrichbook, LeeSM}

\printbibliography

\end{document}